\documentclass[11pt]{article}
\usepackage{tikz}
\usetikzlibrary{fit,positioning}
\usepackage[english]{babel}
\usepackage{amsmath}
\usepackage{amsthm}
\usepackage{amssymb}
\usepackage{mathrsfs}
\usepackage{amscd}
\usepackage{graphicx}
\usepackage{xcolor}
\usepackage[margin=1in]{geometry}

\def\ex{\mbox{ex}}

\newtheorem{thm}{Theorem}[section]

\newtheorem{lem}[thm]{Lemma}
\newtheorem{prop}[thm]{Proposition}

\newtheorem{problem}[thm]{Problem}
\newtheorem{clm}{\bf Claim}[section]

\begin{document}
\title{Turán Numbers of Linear Stars and Expanded Trees}
\author{Xinmin Hou$^{a,b}$, Anyao Wang$^c$\\
\small $^a$ School of Mathematical Sciences,\\
\small University of Science and Technology of China, Hefei 230026, Anhui, China\\
\small $^b$ Hefei National Laboratory,\\
\small University of Science and Technology of China, Hefei 230088, Anhui, China\\
\small $^c$School of Mathematics and Statistics,\\
\small Anhui Normal University, Wuhu, Anhui, China}
\maketitle
\date{ }
\begin{abstract}
For an $r$-uniform hypergraph $\mathcal H$, let
$\ex_r(n,\mathcal H)$ denote the maximum number of edges in an
$n$-vertex $r$-uniform hypergraph containing no copy of $\mathcal H$.
The $r$-uniform expansion $T^{(r)}$ of a tree $T$ is obtained by
adding $r-2$ new vertices to each edge, with the added sets pairwise
disjoint and disjoint from $V(T)$. In particular, the expansion of
$K_{1,k}$ is the linear star $\mathcal L_k^r$, whose $k$ edges
intersect pairwise in a single common vertex.
Chung and Frankl (1987) determined $\ex_3(n,\mathcal L_k^3)$ and all
extremal hypergraphs for every $k\ge3$ and sufficiently large $n$.
We extend the odd-$k$ part of their result to every $r\ge5$,
determining the exact Tur\'an number and the unique extremal
hypergraph up to isomorphism for sufficiently large $n$.
More generally, we extend this result to expansions of trees for
which the minimum degree among the vertices in the smaller color
class is odd and at least three. For every $r\ge5$ and sufficiently
large $n$, we establish lower and upper bounds on their Tur\'an
numbers. When no two vertices attaining this minimum degree in
that color class have a common neighbor, we determine the exact
Tur\'an number and the unique extremal hypergraph up to isomorphism.
The proofs combine the delta-system method with stability arguments and a detailed analysis of extremal constructions. 
\end{abstract}

%\begin{abstract}
%Let $T$ be a tree and let $r\ge 2$. The $r$-uniform expansion
%$T^{(r)}$ is obtained by enlarging every edge of $T$ with a distinct
%set of $r-2$ new vertices. Thus $T^{(r)}$ is a linear $r$-uniform
%hypergraph. Exact Tur\'an numbers are known for several expanded paths
%and cycles, whereas for general expanded trees the earlier results are
%primarily asymptotic or structural.
%{\color{blue}For $r\ge5$, we obtain bounds when the minimum degree
%$k$ in the smaller color class is odd. We determine the exact number
%and the unique extremal construction when no two minimum-degree
%vertices in that class have a common neighbor. We also give
%a counterexample with $k=4$ showing that an even-$k$ extension
%requires additional core traces.}
%\end{abstract}

\section{Introduction}

A {\it hypergraph} $\mathcal{F}=(V, E)$ consists of a vertex set $V$ and an edge set $E$, where every edge is a subset of $V$. The hypergraph $\mathcal{F}$ is said to be  {\it $r$-uniform} (or an {\it $r$-uniform system}) if each edge contains exactly $r$ vertices. When the vertex set is $V=[n]=\{1,2,\ldots, n\}$, we denote an $r$-uniform hypergraph (or $r$-graph for short) on $[n]$  by $\mathcal{F}\subseteq {[n]\choose r}$.
For simplicity, we write $\mathcal{F}$ for $E(\mathcal{F})$,  and its size is denoted by $|\mathcal{F}|(=|E(\mathcal{F})|)$.

For a hypergraph $\mathcal G$ and a positive integer $q$, its
\emph{$q$-shadow} is
\[
 \partial_q(\mathcal G)=
 \left\{Q\in{V(\mathcal G)\choose q}:
 Q\subseteq E\text{ for some }E\in\mathcal G\right\}.
\]
Following~\cite{FJ15}, an $r$-graph $\mathcal T$ is an
\emph{$r$-tree} if its edges admit an ordering $E_1,\ldots,E_m$
such that, for every $i>1$, there is an index $\alpha(i)<i$ with
\[
 E_i\cap\left(\bigcup_{j<i}E_j\right)\subseteq E_{\alpha(i)}.
\]
This definition does not require $\mathcal T$ to be connected.
An $r$-graph is \emph{$2$-reducible} if every edge contains at least
two vertices that belong to no other edge of the hypergraph.
A hypergraph is \emph{embeddable} in another hypergraph if there is
an injective map between their vertex sets that maps every edge of
the first hypergraph to an edge of the second.

Given a graph $T$ and an integer $r\ge 2$, the \emph{$r$-uniform
expansion} of $T$, denoted by $T^{(r)}$, is the $r$-graph obtained by
replacing every edge $e\in E(T)$ with $e\cup P_e$, where the sets
$P_e$ have size $r-2$ and are pairwise disjoint and disjoint from
$V(T)$.  When $T$ is a tree, we call $T^{(r)}$ an \emph{expanded
tree}; it is also a linear $r$-tree.
{F\"uredi~\cite{Fu14} studied linear trees in uniform
hypergraphs and established their basic properties.}
For $r\ge4$, $T^{(r)}$ is $2$-reducible, since every
expanded edge contains $r-2\ge2$ vertices belonging to no other edge.

  A {\it vertex cover}  of $\mathcal{F}$ is a subset of $V$ intersecting all edges of $\mathcal{F}$. 
The  minimum size of such a cover, called the {\it cover number},  is denoted by $\tau(\mathcal{F})$. 
Given an integer $t\ge 1$, a {\it $t$-crosscut} of $\mathcal{F}$ is a subset $X\subseteq V(\mathcal{F})$ satisfying that $|e\cap X|=t$ for all $e\in \mathcal{F}$. Note that not all hypergraphs admit a $t$-crosscut; for instance, $K_4^3$ and $K_4^3-e$ contains no 2-crosscut. 
When $\mathcal{F}$ has at least one $t$-crosscut, define $\sigma_t(\mathcal{F})$ (the {\it $t$-crosscut number}) to be the minimum size of such a $t$-crosscut. For brevity, we write crosscut and $\sigma(\mathcal{F})$ for the case $t=1$.  A fundamental relation is $\tau(\mathcal{F})\le \sigma_t(\mathcal{F})$.

Given two $r$-graphs $\mathcal{F}$ and $\mathcal{H}$, we say $\mathcal{F}$ is $\mathcal{H}$-free if $\mathcal{F}$ does not contain a copy of $\mathcal{H}$ as a subgraph. The {\em Tur\'an number} $\ex_r(n, \mathcal{H})$ is defined to be  the maximum number of hyperedges in an $\mathcal{H}$-free $r$-graph on $n$ vertices, i.e., 
$$\ex_r(n, \mathcal{H})=\max\{|\mathcal{F}| :   \mathcal{F}\text{ is }\mathcal{H}\text{-free and } |V(\mathcal{F})|=n\}.$$
We will omit the index $r$ if the uniformity $r$ is clear from context. 
 
{For a hypergraph $\mathcal H$ and $X\subseteq V(\mathcal H)$,
the \emph{link of the set $X$} is
\[
 \operatorname{L}_{\mathcal H}(X)
 =\{Y\subseteq V(\mathcal H)\setminus X:X\cup Y\in\mathcal H\}.
\]
The degree of $X$ is $\deg_{\mathcal H}(X)=|\operatorname{L}_{\mathcal H}(X)|$;
for a vertex $v$, we also write $d_{\mathcal H}(v)=\deg_{\mathcal H}(\{v\})$.
For a graph $G$, we reserve $N_G$ for the ordinary vertex neighborhood:
\[
 N_G(v)=\{u:uv\in E(G)\},\qquad
 N_G(U)=\left(\bigcup_{u\in U}N_G(u)\right)\setminus U.
\]
Thus a graph neighborhood is a set of vertices, whereas a link is a family of sets.}
When $|X|=k$, the degree of $X$ is specifically termed the $k$-degree of $X$. Denote   $\delta_k(\mathcal{H})$ and $\Delta_k(\mathcal{H})$  as the minimum $k$-degree and maximum $k$-degree of $\mathcal{H}$, respectively.
As usual, we write graph for 2-graph and write $\delta(\mathcal{H})$ and $\Delta(\mathcal{H})$ for $\delta_1(\mathcal{H})$ and $\Delta_1(\mathcal{H})$, respectively.
For a subset $X\subseteq V(\mathcal{H})$, let 
$$\delta_{\mathcal{H}}(X)=\min\{d_{\mathcal{H}}(x) : x\in X\} \text{ and } \Delta_{\mathcal{H}}(X)=\max\{d_{\mathcal{H}}(x): x\in X\}.$$
Let $\mathcal{H}[X]$ be the subgraph of $\mathcal{H}$ induced by $X$ for a subset $X\subseteq V(\mathcal{H})$ or $X\subseteq  E(\mathcal{H})$.
For $S, T\subset V(\mathcal{H})$, let $e_\mathcal{H}(S, T)$ be the number of hyperedges $e\in \mathcal{H}[S\cup T]$ with $e\cap S\not=\emptyset$ and $e\cap T\not=\emptyset$, if $S=T$, we use $e_\mathcal{H}(S)$ instead of $e_\mathcal{H}(S, S)$, and  $e_\mathcal{H}(u, T)$ instead of $e_\mathcal{H}(\{u\}, T)$ if $S=\{u\}$ for convenience, the index $\mathcal{H}$ will be omitted if no confusion from the context.
A {\it matching} $M$ in $\mathcal{H}$ is a subset of $\mathcal{H}$ with $\delta(\mathcal{H}[M])=\Delta(\mathcal{H}[M])=1$. The {\it matching number} of $\mathcal{H}$, denoted by $\nu(\mathcal{H})$, is the maximum size of a matching in $\mathcal{H}$.

The extremal families for expanded trees have special constructions.
To describe them, we need some definitions and notation. A graph is
called {\it almost $d$-regular} if all vertices, except possibly one
vertex of degree $d-1$, have degree $d$. Denote by $R(n,d)$ the family
of almost $d$-regular graphs on $n$ vertices.
Let $K[U, V]$ be a complete bipartite graph with vertex partitions $U$ and $V$.

\noindent{\bf Construction A:} 
Given an $n$-set $V$ and a subset $W\subseteq V$ with $|W|=a-1$. Let 
$$S_W=\left\{F\in {V\choose r}: F\cap W\neq\emptyset\right\}.$$
We construct an almost $(k-1)$-regular graph $G_1$ as follows:
$$V(G_1)=\{u_1,\ldots,u_{k-1},v_1,\ldots,v_{k-1}, v_k\}\subseteq V\setminus W,$$
and $E(G_1)=(E(K)-M_1)\cup E(v_k)$, where  $K=K[\{u_1, \ldots, u_{k-1}\}, \{v_1,\ldots, v_{k-1}\}]$, and $E(v_k)=\{u_iv_k, v_iv_k : 1\le i\le \lfloor\frac{k-1}{2}\rfloor\}$ and $M_1=\{u_iv_i  : 1\le i\le\lfloor\frac{k-1}{2}\rfloor\}$.
It can be easily checked that $G_1\in R(2k-1, k-1)$ with $v_k$ of degree $k-2$ when $k$ is even.
{Define $G_2=2K_k$ for odd $k$ on vertex set
$V'\subseteq V\setminus W$; all constructions with index $2$ below
are restricted to odd $k$. For each integer $0\le x\le a-2$, fix
a graph $R_x\in R(a-1,x)$ on $W$. Then $\Delta(R_x)\le x$ and
$e(R_x)=\lfloor(a-1)x/2\rfloor$.}
For $i=1,2$, let
$$\mathcal{F}_i=\left\{ F\in{V\setminus W\choose r} : F\cap V(G_i)\not=\emptyset \text{ and $\delta\left(G_i[F\cap V(G_i)]\right)\ge 1$}\right\}.$$
For even $k$ and $\frac k2\le j\le k-1$, let $E_j$ be the family of sets $F\in{V\backslash W\choose r}$ such that $\{ u_j, v_{k-1}, v_k\}\subseteq F$ and $v_k$ is an isolated vertex in $G_1[F\cap V(G_1)]$. Define
$$\mathcal{F}_3=\begin{cases}
\emptyset& k\text{ is odd}\\
\bigcup\limits_{j=\frac{k}{2}}^{k-1} E_j & k \text{ is even}
\end{cases}.$$
For such a choice of $R_x$, let
\[
 S_W(R_x)=\left\{F\in{V\choose r}:
 \begin{array}{l}
 F\cap W\ne\emptyset,\ \text{and either }|F\cap W|=1\\
 \text{or }e(R_x[F\cap W])\ge1
 \end{array}\right\}.
\]
Define $\mathcal L_i(n,r,k,a)=S_W\cup\mathcal F_i\cup\mathcal F_3$
for $i=1,2$, and
\[
 \mathcal L_2(n,r,k,a;R_x)=S_W(R_x)\cup\mathcal F_2
 \qquad(k\text{ odd}).
\]
%The semicolon records the chosen graph: different members of $R(a-1,x)$ need not give restricted constructions of the same size.
%Thus $\mathcal{L}_1$ is governed outside $W$ by the almost regular
%graph $G_1$, while $\mathcal{L}_2$ is governed by two disjoint%
%$k$-cliques; the additional family $\mathcal{F}_3$ corrects the
%even-$k$ construction.

\begin{samepage}
We call the specified graph $H$ (namely $G_i$ above)
the \emph{core graph}, and $S=V(H)$ its \emph{core vertex set}.
The \emph{trace} of a set $F$ on a vertex set $X$ is $F\cap X$;
for a hypergraph $\mathcal F$, its trace on $X$ is
$\mathcal F|X=\{F\cap X:F\in\mathcal F\}$, with repeated sets removed.
The \emph{core trace} of $F$ means $F\cap S$, a vertex set, not
necessarily a graph edge. Graph properties of this trace refer to
the induced graph $H[F\cap S]$. 
%These construction-specific notions are distinct from the center and kernel defined in Section~2.
\end{samepage}

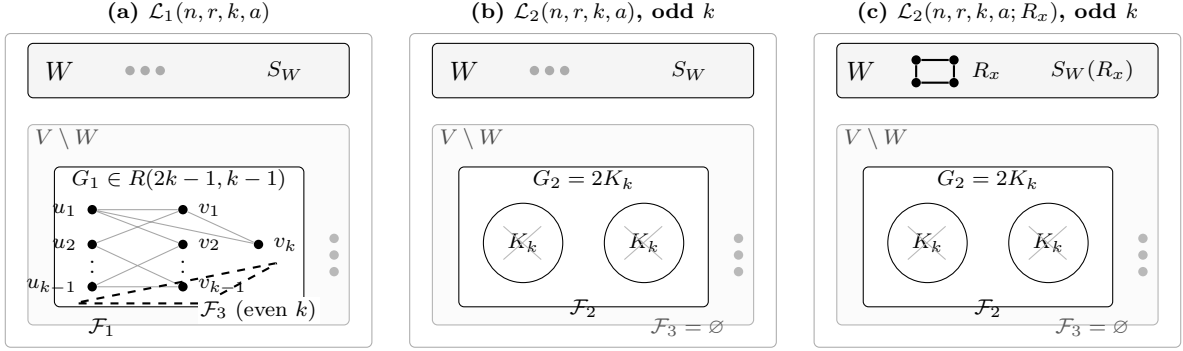
\begin{figure}[htbp]
\centering
\begin{tikzpicture}[
  x=1cm,y=1cm,
  vertex/.style={circle,fill=black,inner sep=1.25pt},
  core/.style={draw=black,fill=white,rounded corners=2pt},
  wblock/.style={draw=black,fill=gray!8,rounded corners=2pt},
  outside/.style={draw=gray!60,fill=gray!3,rounded corners=2pt},
  every node/.style={font=\scriptsize}
]
\begin{scope}
  \draw[gray!65,rounded corners=2pt] (0,0.20) rectangle (4.85,4.35);
  \node[font=\scriptsize\bfseries] at (2.425,4.62)
    {(a) $\mathcal L_1(n,r,k,a)$};
  \node[wblock,minimum width=4.25cm,minimum height=0.72cm] (Wa) at (2.425,3.86) {};
  \node[font=\small] at (0.72,3.86) {$W$};
  \node[black] at (3.70,3.86) {$S_W$};
  \foreach \x in {1.65,1.85,2.05} \node[vertex,fill=gray!65] at (\x,3.86) {};
  \node[outside,minimum width=4.25cm,minimum height=2.65cm] at (2.425,1.82) {};
  \node[black!70] at (0.82,2.96) {$V\setminus W$};
  \node[core,minimum width=3.30cm,minimum height=1.85cm] (G1box) at (2.30,1.66) {};
  \node[black,font=\scriptsize\bfseries] at (2.30,2.42)
    {$G_1\in R(2k-1,k-1)$};
  \node[vertex,label=left:$u_1$]  (u1) at (1.15,2.02) {};
  \node[vertex,label=left:$u_2$]  (u2) at (1.15,1.56) {};
  \node at (1.15,1.30) {$\vdots$};
  \node[vertex,label=left:$u_{k-1}$] (u3) at (1.15,1.00) {};
  \node[vertex,label=right:$v_1$] (v1) at (2.35,2.02) {};
  \node[vertex,label=right:$v_2$] (v2) at (2.35,1.56) {};
  \node at (2.35,1.30) {$\vdots$};
  \node[vertex,label=right:$v_{k-1}$] (v3) at (2.35,1.00) {};
  \node[vertex,label=right:$v_k$] (vk) at (3.35,1.56) {};
  \draw[gray!70] (u1)--(v1) (u1)--(v2) (u2)--(v1)
    (u2)--(v3) (u3)--(v2) (u3)--(v3) (vk)--(u1) (vk)--(v1);
  \draw[black,dashed,thick,rounded corners=2pt]
    (0.90,0.78)--(2.62,0.78)--(3.62,1.32)--cycle;
  \node[black,fill=white,inner sep=1pt] at (3.35,0.69)
    {$\mathcal F_3$ (even $k$)};
  \node at (1.28,0.47) {$\mathcal F_1$};
  \foreach \y in {1.20,1.42,1.64} \node[vertex,fill=gray!55] at (4.35,\y) {};
\end{scope}

\begin{scope}[xshift=5.35cm]
  \draw[gray!65,rounded corners=2pt] (0,0.20) rectangle (4.85,4.35);
  \node[font=\scriptsize\bfseries] at (2.425,4.62)
    {(b) $\mathcal L_2(n,r,k,a)$, odd $k$};
  \node[wblock,minimum width=4.25cm,minimum height=0.72cm] (Wb) at (2.425,3.86) {};
  \node[font=\small] at (0.72,3.86) {$W$};
  \node[black] at (3.70,3.86) {$S_W$};
  \foreach \x in {1.65,1.85,2.05} \node[vertex,fill=gray!65] at (\x,3.86) {};
  \node[outside,minimum width=4.25cm,minimum height=2.65cm] at (2.425,1.82) {};
  \node[black!70] at (0.82,2.96) {$V\setminus W$};
  \node[core,minimum width=3.30cm,minimum height=1.85cm] at (2.30,1.66) {};
  \node[black,font=\scriptsize\bfseries] at (2.30,2.42)
    {$G_2=2K_k$};
  \node[draw=black,circle,fill=white,minimum size=1.05cm] (Kone) at (1.50,1.58) {$K_k$};
  \node[draw=black,circle,fill=white,minimum size=1.05cm] (Ktwo) at (3.10,1.58) {$K_k$};
  \draw[gray!55] (1.24,1.36)--(1.74,1.80) (1.24,1.80)--(1.74,1.36)
    (2.84,1.36)--(3.34,1.80) (2.84,1.80)--(3.34,1.36);
  \node at (2.30,0.72) {$\mathcal F_2$};
  \node[black!65] at (3.68,0.46) {$\mathcal F_3=\varnothing$};
  \foreach \y in {1.20,1.42,1.64} \node[vertex,fill=gray!55] at (4.35,\y) {};
\end{scope}

\begin{scope}[xshift=10.70cm]
  \draw[gray!65,rounded corners=2pt] (0,0.20) rectangle (4.85,4.35);
  \node[font=\scriptsize\bfseries] at (2.425,4.62)
    {{(c) $\mathcal L_2(n,r,k,a;R_x)$, odd $k$}};
  \node[wblock,minimum width=4.25cm,minimum height=0.72cm] (Wc) at (2.425,3.86) {};
  \node[font=\small] at (0.62,3.86) {$W$};
  \node[vertex] (r1) at (1.35,4.00) {};
  \node[vertex] (r2) at (1.85,4.00) {};
  \node[vertex] (r3) at (1.35,3.70) {};
  \node[vertex] (r4) at (1.85,3.70) {};
  \draw[black,thick] (r1)--(r2) (r2)--(r4) (r4)--(r3) (r3)--(r1);
  \node[black] at (2.28,3.86) {$R_x$};
  \node[black] at (3.68,3.86) {$S_W(R_x)$};
  \node[outside,minimum width=4.25cm,minimum height=2.65cm] at (2.425,1.82) {};
  \node[black!70] at (0.82,2.96) {$V\setminus W$};
  \node[core,minimum width=3.30cm,minimum height=1.85cm] at (2.30,1.66) {};
  \node[black,font=\scriptsize\bfseries] at (2.30,2.42)
    {$G_2=2K_k$};
  \node[draw=black,circle,fill=white,minimum size=1.05cm] at (1.50,1.58) {$K_k$};
  \node[draw=black,circle,fill=white,minimum size=1.05cm] at (3.10,1.58) {$K_k$};
  \draw[gray!55] (1.24,1.36)--(1.74,1.80) (1.24,1.80)--(1.74,1.36)
    (2.84,1.36)--(3.34,1.80) (2.84,1.80)--(3.34,1.36);
  \node at (2.30,0.72) {$\mathcal F_2$};
  \node[black!65] at (3.68,0.46) {$\mathcal F_3=\varnothing$};
  \foreach \y in {1.20,1.42,1.64} \node[vertex,fill=gray!55] at (4.35,\y) {};
\end{scope}
\end{tikzpicture}
\caption{Schematic views of the three constructions.  
%The shaded block at the top is $W$, the gray block below it is $V\setminus W$, and the
%outlined subgraph governs members lying entirely outside $W$; the gray
%dots are unrestricted filler vertices.  
An $r$-set outside $W$ lies in
$\mathcal F_i$ precisely when the vertices it selects from $G_i$
induce no isolated vertex.  In (a), the dashed triple represents
the additional family $\mathcal F_3$ for even $k$; in (c), $S_W(R_x)$
imposes the displayed $R_x$ condition inside $W$. The drawing of
$R_x$ uses $R_2=C_4\in R(4,2)$ as an example. The TikZ code for this figure was prepared with assistance
from GPT-6 Astra and checked by the authors.}
\label{fig:extremal-constructions}
\end{figure}

\begin{samepage}
{ Given a set $W$ and a positive integer $s$, a {\it $\Delta$-system} or {\it sunflower} of size $s$ is a family of sets $F_1,\ldots, F_s$  such that  $F_i\cap F_j=W$ for all $1\le i<j\le s$,  we call $W$ the kernel and sets $F_1, \ldots, F_s$ the petals (or members) of the $\Delta$-system. A sunflower $\{F_1, \ldots, F_s\}$ is called {\em $r$-uniform} if $|F_i|=r$ for all $1\le i\le s$.}
Let $\mathcal{L}_k^r$ be an $r$-uniform sunflower of size $k$ with kernel of size $1$ (also called a linear star).
\par\end{samepage}
More generally, for $0\le\ell<r$, let
$\mathcal L_{k,\ell}^r$ be an $r$-uniform sunflower with $k$ edges
and a kernel of size $\ell$, and put
\[
 f(n,r,\ell,k)=\ex_r(n,\mathcal L_{k,\ell}^r).
\]
We also use the finite sunflower extremal function
\[
 \phi(r,s)=\max\{|\mathcal D|:\mathcal D\text{ is an }r\text{-graph
 containing no sunflower of size }s\}.
\]
 Chung and Frankl~{\cite{CF87}} determined the exact value of the Tur\'an function of $\mathcal{L}_k^3$ and its extremal graphs.
\begin{thm}[Theorems 1.1 and 1.2 in~{\cite{CF87}}]\label{THM: even k}
	Suppose that $k\ge 3$ and $n>k(k-1)(5k+2)/2$. Then
	$$\ex_3(n, \mathcal{L}_k^3)=f(n,3,1,k)=\begin{cases}
		(n-2k)k(k-1)+2{k\choose	3} & \text{if } k \text{ is odd},\\
		nk(k-\frac{3}2)-\frac 12 (2k^3 - 9k + 6)& \text{if } k \text{ is even}.
	\end{cases}$$
Moreover, equality holds if and only if the extremal
$3$-graph is isomorphic to $\mathcal{L}_2(n,3,k,1)$ when
$k$ is odd, and to $\mathcal{L}_1(n,3,k,1)$ when $k$ is even.
\end{thm}

{Buci\'c, Dragani\'c, Sudakov, and Tran~\cite{BDS21} estimated the value of the Tur\'an function of $\mathcal{L}_k^4$ and got the following result:
For $2\le k\le n$,
$$\ex_4(n, \mathcal{L}_k^4)\approx k^2n^2.$$}  

In this article, we first extend the odd-$k$ part of Theorem~\ref{THM: even k} to all $r\ge5$.  
The example with $k=4$ in the last section shows that the even-$k$
lifting requires additional higher traces; its exact form remains open.

\begin{thm}\label{LEM: linear star}
	Suppose that $r\ge5$, $k\ge3$ is odd, and $n$ is sufficiently
	large.  Then
	\[
	 \ex_r(n,\mathcal L_k^r)
	 =|\mathcal L_2(n,r,k,1)|.
	\]
	Moreover, equality holds if and only if the extremal $r$-graph is
	isomorphic to $\mathcal L_2(n,r,k,1)$.
\end{thm}

\noindent  {\bf Remark A:}   (1) Theorem~\ref{THM: asymptotic} below, specialized to a
one-vertex kernel, determines only the leading term
$(\phi(2,k)+o(1)){n-2\choose r-2}$, whereas
Theorem~\ref{LEM: linear star} determines the exact value,
including all lower-order terms, and identifies the extremal family
when $k$ is odd.  

(2) For comparison, Keevash, Lifshitz, Long, and Minzer~\cite{KLLM25}
consider the $q$-uniform expansion $G^+(q)$ of a bounded-degree
uniform hypergraph $G$.  Write $s=e(G)$ and let $\sigma(G)$ be the
minimum crosscut size.  Their asymptotic theorem states that, for
$\sigma(G)\ge2$ and every $\varepsilon>0$, there is a constant $C=C(q,\Delta(G),\varepsilon)$
such that
\[
 \ex_q(n,G^+(q))=(1\pm\varepsilon)
 \left|\left\{F\in{[n]\choose q}:|F\cap J|=1\right\}\right|
 \qquad\text{whenever } C\le q\le \frac{n}{Cs},
\]
where $J$ is any fixed set of size $\sigma(G)-1$.  Thus, in this
range, the crosscut construction determines the leading term and a
corresponding stability theorem shows that every nearly extremal
family is close to a family controlled by such a bounded set $J$.

They obtain an exact result under an additional criticality
assumption.  Namely, $G$ is called \emph{critical} if it has an edge
$e$ for which
\[
 \sigma(G-e)=\tau(G-e)<\tau(G)=\sigma(G),
\]
where $\tau(G)$ is the vertex-cover number.  For such $G$, in the same
range $C\le q\le n/(Cs)$, their theorem gives
\[
 \ex_q(n,G^+(q))=
 \left|\left\{F\in{[n]\choose q}:F\cap J\ne\emptyset\right\}\right|,
 \qquad |J|=\sigma(G)-1.
\]

These results do not imply Theorem~\ref{LEM: linear star}.  Indeed,
$\mathcal L_k^r$ is the $r$-uniform expansion of the star
$K_{1,k}$.  For $k\ge2$ one has
$\sigma(K_{1,k})=\tau(K_{1,k})=1$, and deleting a single edge still
leaves vertex-cover number one.  Hence $K_{1,k}$ neither satisfies
the hypothesis $\sigma(G)\ge2$ in the general asymptotic theorem nor
is critical in the above sense.  Moreover, the general theorems
require the expansion uniformity to satisfy the unspecified lower
bound $q\ge C$, while Theorem~\ref{LEM: linear star}
applies to every fixed $r\ge5$ in the odd-$k$ case.  It resolves the
first nonzero term, of order $n^{r-2}$, together with all lower-order
terms and the complete extremal classification in that case.
Likewise, the main
theorem below assumes structural conditions on a color class of the
base tree, rather than the above edge-critical identity, and its
extremal construction depends on the local sets $A_0$ and $B_0$.

(3) Furthermore, for odd $k$ we extend this result to the
case where two vertex-disjoint linear $(k-1)$-stars are also
forbidden.  The expansion results of star-forest were given by
Khormali and Palmer~\cite{KP22}, which can be written in our notation as
\[
 \ex_r(n,m\mathcal L_k^r)
 =\binom{n}{r}-\binom{n-m+1}{r}
  +\ex_r(n-m+1,\mathcal L_k^r),
\]
where the parameters are fixed and $n$ is sufficiently large.

%The even case is discussed in Section~\ref{SEC: remarks}.

We further extend Theorem~\ref{LEM: linear star} to the case where two vertex-disjoint linear $(k-1)$-stars are also forbidden. 

\begin{thm}\label{LEM: two linear star}

	Suppose that $r\ge5$, $k\ge3$ is odd, and $n$ is sufficiently
	large. Then
	\[
	 \ex_r\bigl(n,\{\mathcal L_k^r,2\mathcal L_{k-1}^r\}\bigr)
	 =|\mathcal L_1(n,r,k,1)|.
	\]
	Moreover, equality holds if and only if the extremal $r$-graph is
	isomorphic to $\mathcal L_1(n,r,k,1)$.

\end{thm}

{The following result extends
Theorem~\ref{LEM: linear star} to expanded trees. Its exact conclusion
requires $B_0=\emptyset$; the other cases presently give bounds.}

\begingroup

\begin{thm}\label{THM: Main}
Let $r\ge5$, and let $T$ be a tree with color classes $A,B$,
where $a=|A|\le |B|$. Suppose that
$k=\min_{x\in A}d_T(x)\ge3$ is odd. Define
\[
 A_0=\{x\in A:d_T(x)=k\},\qquad
 B_0=\{y\in B:|N_T(y)\cap A_0|\ge2\}.
\]
For all sufficiently large $n$,
\begin{equation}\label{EQ: tree odd bounds}
 |\mathcal L_1(n,r,k,a)|
 \le \ex_r(n,T^{(r)})
 \le |\mathcal L_2(n,r,k,a)|.
\end{equation}

If $B_0=\emptyset$, then
\[
 \ex_r(n,T^{(r)})=|\mathcal L_2(n,r,k,a)|,
\]
and the unique extremal $r$-graph, up to isomorphism, is
$\mathcal L_2(n,r,k,a)$.

If $B_0\ne\emptyset$, put
\[
 b=\min_{y\in B_0}d_T(y)-2.
\]
Since $d_T(y)\le a$ for $y\in B$, we have $0\le b\le a-2$.
When $b\ge1$, for every choice of $R_{b-1}\in R(a-1,b-1)$ on $W$,
the lower bound in \eqref{EQ: tree odd bounds} can be strengthened to
\[
 \ex_r(n,T^{(r)})\ge
 \max\bigl\{|\mathcal L_1(n,r,k,a)|,
             |\mathcal L_2(n,r,k,a;R_{b-1})|\bigr\}.
\]
\end{thm}
\endgroup

{\noindent\textbf{Remark B.}
The exact case includes stars $K_{1,k}$ with odd $k$, and the tree
obtained by identifying one leaf of $K_{1,k}$ with one leaf of
$K_{1,k+1}$: its smaller color class consists of the two centers,
and only one of them has minimum degree, so $B_0=\emptyset$.
} %

%{\color{blue}The rest of the paper is organized as follows.
%Section~2 collects the homogeneous-family and centralized-family
%tools used in the upper bound. Section~\ref{SEC: proof main} proves
%Theorem~\ref{THM: Main}, first verifying the lower-bound
%constructions and then establishing the stability-based upper bound.
%Section~\ref{SEC: remarks} presents a counterexample for $k=4$
%and three open problems.}

The rest of the paper is organized as follows.
Section~2 collects the necessary tools from the $\Delta$-system
method and proves Theorems~\ref{LEM: linear star}
and~\ref{LEM: two linear star}.
Section~\ref{SEC: proof main} proves Theorem~\ref{THM: Main},
first verifying the lower-bound constructions and then establishing
the upper bound and the equality characterization.
Section~\ref{SEC: remarks} presents a counterexample for $k=4$
and poses open problems concerning the remaining cases for linear
stars and expanded trees.

\section{Centralized families and the proof of Theorems~\ref{LEM: linear star} and \ref{LEM: two linear star}}

In this section, we use $\Delta$-system method to reduce the problem of embedding a linear tree and its subgraphs into a hypergraph $\mathcal{F}$ to one where $\mathcal{F}$ belong to the so-called centralized family.

Given a hypergraph $\mathcal{F}$ and a subset $W\subseteq V(\mathcal{F})$, the {\it kernel degree} of $W$, denoted by $\deg^*_{\mathcal{F}}(W)$, is
$$\deg^*_{\mathcal{F}}(W) = \max\{s : \text{there is a $\Delta$-system of size $s$ with kernel $W$ in $\mathcal{F}$}\}.$$
%We also write $\deg_{\mathcal F}^*=\deg_{\mathcal F}^*(W)$.
In other words, $\deg^*_{\mathcal{F}}(W)$ is the matching number of $\{F\setminus W : W\subseteq F\in E(\mathcal{F})\}$.
An  $r$-graph $\mathcal{F}$ is called {\it $r$-partite} if there is a partition $(X_1,\ldots, X_r)$, called an $r$-partition, of the vertex set $V(\mathcal{F})$ such that every edge  contains exactly one element of each part $X_i$.
For $F\in\mathcal F$ and $I\subseteq[r]$, write
\[
 F\langle I\rangle=F\cap\bigcup_{i\in I}X_i,\qquad
 \mathcal F\langle I\rangle=\{F\langle I\rangle:F\in\mathcal F\}.
\]
These are projections onto the parts indexed by $I$; write
$F\langle i\rangle=F\langle\{i\}\rangle$ for a single part.
%In contrast, square brackets denote induced restriction on a vertex set:
%$\mathcal F[Z]=\{F\in\mathcal F:F\subseteq Z\}$.}

\begin{lem}[The homogeneous subfamily lemma, see~\cite{Fu83} and~\cite{FJ15}]\label{LEM: homogeneous family}
	For any positive integers $s$ and $r$, there is a positive constant $c=c(r,s)$ such that for every family $\mathcal{F}\subseteq {{[n]}\choose r}$, there exist $\mathcal{F}^*\subseteq\mathcal{F}$ with $|\mathcal{F}^*|\ge c|\mathcal{F}|$ and some $\mathcal{J}\subseteq 2^{[r]}\setminus\{[r]\}$ (called the intersection pattern) such that:
		
		(1) $\mathcal{F}^*$ is $r$-partite, together with an $r$-partition $(X_1,\dots,X_r)$.
		
		(2) For every $F\in\mathcal{F}^*$ and $I\in\mathcal{J}$,
		$\deg_{\mathcal{F}^*}^*({F\langle I\rangle})\ge s$.  If
		$I\subsetneq[r]$ and $I\notin\mathcal{J}$, then there is no
		$F'\in\mathcal{F}^*$ such that $F\cap F'={F\langle I\rangle}$.
		
		(3) $\mathcal{J}$ is closed under intersection, i.e., for $I, I'\in\mathcal{J}$, we have $I\cap I'\in\mathcal{J}$ as well.
	
\end{lem}

We call $\mathcal{F}^*$ in Lemma~\ref{LEM: homogeneous family} an {\it $(r,s)$-homogeneous family}  with intersection pattern $\mathcal{J}$.
Given {$\mathcal{J}\subseteq 2^{[r]}\setminus\{[r]\}$}, we define the {\it rank of $\mathcal{J}$}, denoted by $r(\mathcal{J})$, as
$$r(\mathcal{J})=\min\{|D|  : D\subseteq [r], \text{ there is no $B\in\mathcal{J}$ such that } D\subseteq B\}.$$
Let $\mathcal{F}\subseteq {[n]\choose r}$ be an $r$-partite family with an $r$-partition $(X_1,\ldots, X_r)$. Given
any subset $S\subseteq [n]$, its {\it pattern}, denoted by $\Pi(S)$, is defined as
$$\Pi(S) =\{i : S\cap X_i\not=\emptyset\}\subseteq [r].$$
In other words, the pattern of $S$ records which parts in the given $r$-partition that $S$ meets.
F\"uredi, Jiang, Seiver~\cite{FJS14} gave the following proposition, called the rank bound.
\begin{prop}[The rank bound~\cite{FJS14}]\label{PROP: rank bound}
	Let $r,s$ be positive integers and $\mathcal{F}^*$ be an $(r,s)$-homogeneous family on $n$ vertices with intersection pattern $\mathcal{J}$. If $r(\mathcal{J})=p$ then $|\mathcal{F}^*|\le {n\choose p}$.
\end{prop}

We now recall extremal results for the sunflower function
$\phi(r,s)$ defined in the introduction.
Clearly, $\phi(1,s)=s-1$. When $r=2$, Abbott, Hanson, and Sauer~\cite{AHS72} gave the exact value of $\phi(2,s)$.
\begin{thm}[\cite{AHS72}]\label{LEM: AHS}
	Let $G$ be a $\{K_{1,k}, kK_2\}$-free graph. Then
	$$e(G)\le \phi(2, k)=\begin{cases}
		k^2-k & \text{if } k \text{ is odd},\\
		k^2-\frac{3}{2}k & \text{if } k \text{ is even}.
	\end{cases}$$
	Moreover, when $k$ is odd, equality holds if and only
	if $G\cong2K_k$, apart from isolated vertices.  When $k$ is even,
	the displayed bound has, in general, several nonisomorphic equality
	cases; in particular equality holds for
	$G\in R(2k-1,k-1)$ and for
	$G\cong H\sqcup K_{k-1}$ for some $H\in R(k+1,k-1)$,
	where $\sqcup$ denotes vertex-disjoint union.
\end{thm}

The above result was further extended by Wang, Hou, Liu, Ma~\cite{WHLM20} to $\{K_{1,k}, kK_2, 2K_{1,k-1}\}$-free graphs.
\begin{thm}[\cite{WHLM20}]\label{LEM: 2K_{1,k-1}}
	Let $G$ be a $\{K_{1,k}, kK_2, 2K_{1,k-1}\}$-free graph.		
	Then $$e(G)\le g_1(k)=\begin{cases}
		k^2-\frac{3}{2}k+\frac{1}{2} & \text{if } k \text{ is odd},\\
		k^2-\frac{3}{2}k & \text{if } k \text{ is even}.
	\end{cases}$$
	Furthermore, the equality holds if and only if $G\in R(2k-1, k-1)$ when $k$ is odd, and $G\in R(2k-1, k-1)$ or $G\cong H\sqcup K_{k-1}$ for some $H\in R(k+1,k-1)$ when $k$ is even. If we do not care about the trivial components, the optimal graphs are determined completely.
\end{thm}

For general $r\ge 2$, Erd\H{o}s and Rado~\cite{ER60} showed that $\phi(r,s)\le r!(s-1)^r$ and conjectured that $\phi(r,s)=c^r$ for some absolute constant $c=c(s)$ and $s\ge 3$;  Alweiss,  Lovett, Wu, and Zhang~\cite{ALWZ21} improved the upper bound to $(cs^3\log r\log\log r)^r$ for some constant $c$. 

More recently, Brada\v{c}, Buci\'c, and Sudakov~\cite{BBS23} determined the sunflower Tur\'an number up to a constant factor, allowing the number of petals to grow with $n$. In another direction, Kupavskii and Noskov~\cite{KN24} used spread approximations to enlarge the range of parameters for which the Duke--Erd\H{o}s forbidden-sunflower problem is resolved.
Duke and Erd\H{o}s~\cite{DE77} introduced the sunflower
Tur\'an function $f(n,r,\ell,k)$ defined in the introduction.
Frankl and F\"uredi~\cite{FF87} established an asymptotic bound for
$\ex_r(n,\mathcal L_{k,\ell}^r)$.
\begin{thm}[Theorem 2.7 in~\cite{FF87}]\label{THM: asymptotic}
	For fixed integers $r, \ell, k$ with $r\ge 2\ell+3$, we have  $${\ex_r(n, \mathcal{L}_{k,\ell}^r)}=f(n,r,\ell,k)=(\phi(\ell+1,k)+o(1)){n-\ell-1\choose	r-\ell-1}.$$
\end{thm}

Before we prove Theorems~\ref{LEM: linear star} and~\ref{LEM: two linear star}, we introduce more lemmas which will be used in our proofs.
Let $\mathcal{F}\subseteq {[n]\choose r}$ be an $(r, s)$-homogeneous family with a corresponding $r$-partition
$(X_1,\ldots, X_r)$ and intersection pattern $\mathcal{J}$. If there is an $i\in [r]$, such that for any $F\in\mathcal{F}$, $\deg_{\mathcal{F}}(F\backslash {F\langle i\rangle})=1$ but for any $I\subsetneq [r]$ with $i\in I$, we have $\deg^*_{\mathcal{F}} ({F\langle I\rangle}) \ge s$, then we say that $\mathcal{F}$
is {\it homogeneously centralized} with threshold $s$.

\begin{lem}[Lemma 6.5 in~\cite{FJ15}]\label{LEM: centralized}
	Given $r\ge 4$ and a linear $r$-tree $T^{(r)}$  with $s$ vertices. Let $\mathcal{F}\subseteq {[n]\choose r}$ be an $(r,s)$-homogeneous family with a corresponding $r$-partition $(X_1, \ldots, X_r)$ and intersection pattern $\mathcal{J}$. If $T^{(r)}\nsubseteq\mathcal{F}$, then either $|\mathcal{F}|\le {n\choose r-2}$ or $\mathcal{F}$ is homogeneously centralized with threshold $s$.
\end{lem}

\begin{lem}\label{LEM: partition}
	Let $k,r,s$ be positive integers with $k\ge 3$, $r\ge 4$ and $s\ge rk$. If $\mathcal{F}\subseteq {[n]\choose r}$ is $\mathcal{L}_k^r$-free, then $\mathcal{F}$ can be partitioned into subfamilies $\mathcal{G}_1,\ldots,\mathcal{G}_m, \mathcal{F}_0$ such that for any $i\in[m]$, $\mathcal{G}_i$ is $(r,s)$-homogeneous with an intersection pattern $\mathcal{J}_i$ of rank at least $r-2$, and $|\mathcal{F}_0|\le\frac{1}{c(r,s)}{	n\choose r-3}$.
\end{lem}

\begin{proof}
	Set $\mathcal{R}_0=\mathcal{F}$ and apply
	Lemma~\ref{LEM: homogeneous family} repeatedly.  As long as
	$|\mathcal{R}_{j-1}|>c(r,s)^{-1}{n\choose r-3}$, choose an
	$(r,s)$-homogeneous subfamily
	$\mathcal{G}_j\subseteq\mathcal{R}_{j-1}$ with
	$|\mathcal{G}_j|\ge c(r,s)|\mathcal{R}_{j-1}|$, and put
	$\mathcal{R}_j=\mathcal{R}_{j-1}\setminus\mathcal{G}_j$.
	
	By the extraction rule,
	\[
	|\mathcal{G}_j|\ge c(r,s)|\mathcal{R}_{j-1}|
	>{n\choose r-3}.
	\]
	Let $p=r(\mathcal{J}_j)$.  If $p\le r-3$, then
	Proposition~\ref{PROP: rank bound} gives
	\[
	|\mathcal{G}_j|\le {n\choose p}\le {n\choose r-3},
	\]
	a contradiction.  Hence $r(\mathcal{J}_j)\ge r-2$.
	
	The process terminates after finitely many steps.  With
	$\mathcal{F}_0$ equal to the final remainder, we have
	$|\mathcal{F}_0|\le c(r,s)^{-1}{n\choose r-3}$, as required.
\end{proof}

Recall that $\mathcal F|S$ denotes the trace on $S$,
with repeated sets removed, as defined in the introduction.
Define the {\it center} of $\mathcal{F}$ by $C(\mathcal{F})=\{v\in V(\mathcal{F}): d_\mathcal{F}(v)\ge 2\}$, and the {\it kernel} of $\mathcal{F}$ by $K(\mathcal{F})=\cap_{F\in \mathcal{F}} F$. { { Setting $Y=C(\mathcal{F})$, we call  $\mathcal{C}(\mathcal{F})=\mathcal{F}|Y$ the {\it full trace} of $\mathcal{F}$.}
The following result was given by Frankl and F\"uredi~\cite{FF87}.

\begin{thm}[Lemma 7.3 in~\cite{FF87}]\label{THM: full trace}
	Let $\mathcal{A}$ be an $r$-graph, and set
	$p=|K(\mathcal{A})|$ and
	$q=|C(\mathcal{A})|-|K(\mathcal{A})|$. Let $\mathcal{F}$ be an
	$r$-partite $r$-graph with intersection pattern $\mathcal{J}$.
	For each $F\in\mathcal{F}$, define
	\[
		\mathcal{I}(F)=\{{F\langle J\rangle}:J\in\mathcal{J}\}\subseteq 2^F.
	\]
	If $r(\mathcal{J})\ge r-p-1$ and $r\ge 3p+q+2$, then, for every
	$F\in\mathcal{F}$, either $\mathcal{I}(F)$ contains a copy of the
	full trace of $\mathcal{A}$, or there is a
	$B\in {F\choose p+1}$ such that
	\[
		\{E:B\subseteq E\subsetneq F\}\subseteq\mathcal{I}(F).
	\]
\end{thm}

\begin{proof}[Proof of Theorems~\ref{LEM: linear star} and \ref{LEM: two linear star}]
Throughout this proof, $k$ is odd.  Fix $\mathcal F\subseteq{[n]\choose r}$ satisfying
either 
\begin{enumerate}
\item[(I)] $\mathcal F$ is $\mathcal L_k^r$-free and
$|\mathcal F|\ge|\mathcal L_2(n,r,k,1)|$; or
\item[(II)] $\mathcal F$ is
$\{\mathcal L_k^r,2\mathcal L_{k-1}^r\}$-free and
$|\mathcal F|\ge|\mathcal L_1(n,r,k,1)|$;
\end{enumerate}
and write $\mathcal L=\mathcal L_2(n,r,k,1)$ in
case (I) and $\mathcal L=\mathcal L_1(n,r,k,1)$ in case (II).  In
particular, $\mathcal F$ is $\mathcal L_k^r$-free in both cases and
$|\mathcal F|\ge|\mathcal L|$.

	Let $s\ge 2rk$. By Lemma~\ref{LEM: partition}, $\mathcal{F}$ can be partitioned into subfamilies $\mathcal{G}_1,\ldots,\mathcal{G}_m$, $\mathcal{F}_0$ such that for any $i\in[m]$, $\mathcal{G}_i$ is $(r,s)$-homogeneous with an intersection pattern $\mathcal{J}_i$ of rank at least $r-2$, and $|\mathcal{F}_0|\le\frac{1}{c(r,s)}{	n\choose r-3}$.
	For $F\in\mathcal{G}_i$, let
	$\mathcal{I}_i(F)=\{{F\langle J\rangle}:J\in\mathcal{J}_i\}$.
	Applying Theorem~\ref{THM: full trace} to $\mathcal{L}_k^r$ and each $\mathcal{G}_i$, note that $p=|K(\mathcal{L}_k^r)|=1$ and $q=|C(\mathcal{L}_k^r)|-|K(\mathcal{L}_k^r)|=1-1=0$ and $r\ge 5=3p+q+2$, 
	for every $F\in\mathcal{G}_i$, either
	$\mathcal{I}_i(F)$ contains the full trace
	$\mathcal{C}(\mathcal{L}_k^r)$, or
	$\{E:B\subseteq E\subsetneq F\}\subseteq\mathcal{I}_i(F)$ for
	some $B\in {F\choose2}$. The former would put a
	singleton of kernel degree at least $s$ in $\mathcal{G}_i$, thereby
	producing a copy of $\mathcal{L}_s^r$, a contradiction. Hence the
	latter holds for every $F$. Let {$\mathscr B(F)$} be the set of
	all pairs $B\in{F\choose2}$ with this property.
	
	\begin{clm}\label{CL: B(F)}
		For any $ F\in\mathcal{G}_i$ with $i\in[m]$, {$|\mathscr B(F)|=1$}.
	\end{clm}
  \begin{proof}
		Suppose to the contrary that there is an $F\in\mathcal{G}_i$ such that $\{\{v_1, v_2\}, \{v_3, v_4\}\}\subseteq {\mathscr B(F)}$.
		Then $\{v_1,v_2\}\cup\{v_3, v_4\}\subseteq F$. We may assume $v_4\notin \{v_1, v_2\}$. 
		Clearly, $\{v_1, v_2, v_3\}\in\mathcal{I}_i(F)$.
		Since $\{v_3,v_4\}\in\mathcal{I}_i(F)$ and
		$\mathcal{I}_i(F)$ is closed under intersection, we have
		$\{v_3\}\in\mathcal{I}_i(F)$, which leads to a copy of
		$\mathcal{L}_s^r$ with kernel $v_3$ in $\mathcal{G}_i$, a
		contradiction.
   \end{proof}
	
	Let $\mathcal{F}^*=\mathcal{G}_1\cup\ldots\cup\mathcal{G}_m$.
{For $F\in\mathcal F^*$, denote by $B(F)\in\binom F2$
the unique member of $\mathscr B(F)$.}
Define $\mathcal{G}'$ as the 3-graph  with vertex set $[n]$ and edge set being the maximum set of $\{e\in { [n]\choose   3} : d_{\mathcal{F}}^*(e)\ge s\}$.
	\begin{clm}\label{CL: expansion}
		$\mathcal{G'}$	is $\mathcal{L}_k^3$-free. Moreover, if $\mathcal{F}$ is $2\mathcal{L}_{k-1}^r$-free, then $\mathcal{G'}$ is $2\mathcal{L}_{k-1}^3$-free.
	\end{clm}
\begin{proof}
	Suppose to the contrary that there is an $\mathcal{L}_k^3\subseteq\mathcal{G'}$. We show that $\mathcal{F}$ contains a copy of $\mathcal{L}_k^r$. Thus we have a contradiction.
	Let $\mathcal{L}_k^3=\{D_1, \ldots, D_k\}\subseteq\mathcal{G'}$. Then $d_{\mathcal{F}}^*(D_i)\ge s$ for $1\le i\le k$.
	We show that $\mathcal{L}_k^3$ in $\mathcal{G}'$ can be expanded to a copy of $\mathcal{L}_k^r$ in $\mathcal{F}$ by induction on $k$.
	For $k=1$, since $d_{\mathcal{F}}^*(D_1)\ge s\ge1$, there is a
	hyperedge $F_1\in\mathcal{F}$ containing $D_1$, which is the
	required copy of $\mathcal{L}_1^r$.
	Now suppose $k\ge 2$ and any $\mathcal{L}_{k-1}^3\subseteq \mathcal{G}'$ can be expanded to an $\mathcal{L}_{k-1}^r$ in $\mathcal{F}$. 
Let $x$ be the kernel of $\{D_1,\ldots,D_k\}$ and put
\[
V_0=\bigcup_{i=1}^k D_i.
\]
We choose $F_1,\ldots,F_k$ inductively so that
\[
D_i\subseteq F_i,\qquad F_i\cap V_0=D_i,
\]
and the sets $F_i\setminus D_i$ are pairwise disjoint.

Suppose that $F_1,\ldots,F_{t-1}$ have already been chosen.
Since $\deg_{\mathcal F}^*(D_t)\ge s$, there is a
$\Delta$-system $\{F'_1,\ldots,F'_s\}$ with kernel $D_t$.
Its petals $F'_j\setminus D_t$ are pairwise disjoint.
Let
\[
W_t=(V_0\setminus D_t)\cup
\bigcup_{i=1}^{t-1}(F_i\setminus D_i).
\]
Since $|V_0|=2k+1$, we have
\[
|W_t|
\le 2k-2+(t-1)(r-3)
\le (k-1)(r-1)<kr\le s.
\]
As the petals are pairwise disjoint, at most $|W_t|$ of them
intersect $W_t$. Hence some $F'_j\setminus D_t$ is disjoint
from $W_t$. Set $F_t=F'_j$.
Then $F_t\cap V_0=D_t$, and for every $i<t$,
\[
F_t\cap F_i=D_t\cap D_i=\{x\}.
\]
Thus $\{F_1,\ldots,F_k\}$ is a copy of
$\mathcal L_k^r$ in $\mathcal F$.
	
	Now suppose that
	$2\mathcal L_{k-1}^3\subseteq\mathcal G'$.  List its
	$2(k-1)$ triples as $D_1,\ldots,D_{2(k-1)}$, keeping the two
	components consecutive, and put
	$V_0=\bigcup_tD_t$.  Thus $|V_0|=2(2k-1)=4k-2$.
	Choose $F_t\supseteq D_t$ successively.  At step $t$, a petal in
	a sunflower of size $s$ with kernel $D_t$ must avoid
	\[
	 (V_0\setminus D_t)\cup
	 \bigcup_{j<t}(F_j\setminus D_j),
	\]
	whose size is at most
	\[
	 4k-5+(2k-3)(r-3)<2kr\le s.
	\]
	Hence an avoiding petal always exists.  Within each component the
	chosen $r$-edges meet exactly in its original center, and edges
	belonging to different components are disjoint.  They therefore
	form a copy of $2\mathcal L_{k-1}^r$ in $\mathcal F$, as required.
	This completes the proof of the claim.
 \end{proof}

\medskip
Let $D(F)=F\backslash B(F)$ for $F\in \mathcal{F}^*$.
Let $\mathcal{L}'\subseteq {[n]\choose 3}$ with edge set $$\mathcal{B}=\{ B(F)\cup\{y\} : F\in\mathcal{F}^* \text{ and } y\in D(F)\}.$$
Note that $d_{\mathcal{F}}^*(B(F))\ge s$ and $d_{\mathcal{F}}^*(B(F)\cup\{y\})\ge s$ for any $F\in\mathcal{F}^*$ and $y\in D(F)$. 
Therefore, $\mathcal{L}'\subseteq\mathcal{G}'$.
By Claim~\ref{CL: expansion}, $\mathcal{L}'$ is $\mathcal{L}_k^3$-free or $\{\mathcal{L}_k^3, 2\mathcal{L}_{k-1}^3\}$-free (if $\mathcal{F}$ is $\{\mathcal{L}_k^r, 2\mathcal{L}_{k-1}^r\}$-free). By Theorem~\ref{THM: even k}, $|\mathcal{L}'|\le f(n,3,1,k)$.
Define  $\mathcal{L}''$ to be a directed bipartite graph with vertex partition $\{B(F) : F\in \mathcal{F}^*\}$ and $[n]$ whose arc set is
 $$\mathcal{B}''=\{(B(F),y) : F\in\mathcal{F}^* \text{ and } y\in D(F)\}.$$
Note that one edge $B(F)\cup\{y\}\in\mathcal{B}$, say $B(F)=\{a,b\}$, possibly induces three arcs, say $(\{a,b\},y)$, $(\{a,y\},b)$ and $(\{b, y\},a)$, if there are some $F', F''\in\mathcal{F}^*$ with $B(F')=\{a,y\}, b\in D(F')$ and $B(F'')=\{b,y\}, a\in D(F'')$. 
Theorem~\ref{THM: even k} gives 
$|\mathcal L''|\le3|\mathcal L'|\le3f(n,3,1,k)\le cn$ for some constant $c=c(k)$.

\begin{clm}\label{CL: kernel graph}
	$|\mathcal{F}^*|\le\frac{|\mathcal{L'}|}{r-2}{n-3\choose r-3}$.
\end{clm}
  \begin{proof}
	Let $$\mathcal{M}=\{(B(F)\cup\{y\},F) : F\in\mathcal{F}^* \text{ and } y\in D(F)\}.$$
	For fixed $F\in\mathcal{F}^*$, there are $r-2$ choices for $y\in D(F)$ with $(B(F)\cup\{y\}, F)\in \mathcal{M}$.
	For fixed $B(F)\cup\{y\}$, there are at most ${n-3\choose r-3}$ hyperedges $F\in\mathcal{F}^*$ containing $B(F)\cup\{y\}$.
	Therefore,
	$$|\mathcal{F}^*|(r-2)=|\mathcal{M}|\le |\mathcal{L}'|{n-3\choose r-3}. $$
	The claim holds.
  \end{proof}
\medskip

For a 2-set   $B\in {[n]\choose 2}$, let $N^+(B)=\{y\in [n] : (B,y)\in \mathcal{L}''\}$ and let $d^+(B)=|N^+(B)|$ be the out-degree of $B$ in $\mathcal{L}''$.
Let $\mathcal{A}=\{B\in {[n]\choose 2} : {d^+(B)}\le x\}$ and ${\mathcal C}=\{B\in {[n]\choose 2} : d^+(B)> x\}$, where $x=n^{1-\frac{7}{3(r-2)}}$.
Let  $\mathcal{F}_{\mathcal{A}}=\{F\in\mathcal{F}^* : B(F)\in \mathcal{A}\}$. Then $|\mathcal{F}_{\mathcal{A}}|\le|\mathcal{A}|{  x\choose r-2}\le n^{r-\frac{7}{3}}$.
{Let $V{\mathcal C}=\bigcup_{B\in{\mathcal C}}B$.
The subgraph $\mathcal L''[{\mathcal C},V{\mathcal C}]$ induced by
${\mathcal C}\cup V{\mathcal C}$ has arc set
\[
 \mathcal A_{{\mathcal C}}
 =\{(B,y)\in\mathcal B'':B\in{\mathcal C},\ y\in V{\mathcal C}\}.
\]
Define its associated $3$-graph by
\[
 \mathcal Q_{{\mathcal C}}
 =\{B\cup\{y\}:(B,y)\in\mathcal A_{{\mathcal C}}\}
 \subseteq\mathcal L'[V{\mathcal C}].
\]}
Let
$$\mathcal{F}_{{\mathcal{C}}}=\{F\in\mathcal{F}^* : B\cup\{y\}\subseteq F \text{ for some ${(B,y)\in\mathcal A_{{\mathcal C}}}$}\}.$$
{Note that
$|V{\mathcal C}|\le2|{\mathcal C}|\le2|\mathcal L''|/x\le2cn/x$.
The $3$-graph $\mathcal Q_{{\mathcal C}}$ is $\mathcal L_k^3$-free,
and each of its triples represents at most three arcs. Thus
\[
 |\mathcal A_{{\mathcal C}}|\le3|\mathcal Q_{{\mathcal C}}|
 \le c|V{\mathcal C}|\le2c^2n^{\frac{7}{3(r-2)}}.
\]
Consequently, for $r\ge5$,
\[
 |\mathcal F_{{\mathcal C}}|
 \le|\mathcal A_{{\mathcal C}}|{n-3\choose r-3}
 \le2c^2n^{r-3+\frac{7}{3(r-2)}}
 =O(n^{r-\frac{20}{9}}).
\]}
	Let $\ell=\phi(2,k)$ and $\ell'=g_1(k)$. Directly from the
	finite-kernel constructions,
	\[
		|\mathcal{L}|=
		\begin{cases}
		\ell{n-2\choose r-2}+O(n^{r-3}),
			&\text{in the $\mathcal{L}_k^r$-free case},\\
		\ell'{n-2\choose r-2}+O(n^{r-3}),
			&\text{in the $\{\mathcal{L}_k^r,
			2\mathcal{L}_{k-1}^r\}$-free case}.
		\end{cases}
	\]
Let $\tilde{\mathcal{F}}=\mathcal{F}^*\setminus(\mathcal{F}_\mathcal{A}\cup \mathcal{F}_{\mathcal{C}})$. 
	Then
	\[
		\tilde{\mathcal{F}}
		=\{F\in\mathcal{F}^*:  B(F)\in {\mathcal{C}} \text{ and }  F\cap V{\mathcal{C}}=B(F)\}.
	\]
	Indeed, membership outside $\mathcal{F}_{\mathcal A}$ forces
	$B(F)\in{\mathcal C}$, while membership outside
	$\mathcal{F}_{{\mathcal C}}$ forces
	$(F\setminus B(F))\cap V{\mathcal C}=\emptyset$.
Therefore, 
$$|\tilde{\mathcal{F}}|=|\mathcal{F}|-|\mathcal{F}_0|-|\mathcal{F}_{\mathcal{A}}\cup \mathcal{F}_{{\mathcal{C}}}|\ge\ell{n-2\choose r-2}-O(n^{r-3})-O(n^{r-\frac 73})-O(n^{r-\frac{20}{9}})=\ell{n-2\choose r-2}-O(n^{r-\frac{20}{9}})$$
if $\mathcal{F}$ is $\mathcal{L}_k^r$-free,
 and $$|\tilde{\mathcal{F}}|=|\mathcal{F}|-|\mathcal{F}_0|-|\mathcal{F}_{\mathcal{A}}\cup\mathcal{F}_{{\mathcal{C}}}|\ge\ell'{n-2\choose r-2}-O(n^{r-3})-O(n^{r-\frac 73})-O(n^{r-\frac{20}{9}})=\ell'{n-2\choose r-2}-O(n^{r-\frac{20}{9}})$$
if $\mathcal{F}$ is $\{\mathcal{L}_k^r, 2\mathcal{L}_{k-1}^r\}$-free.
Let $\tilde{\mathcal{L}}\subseteq\mathcal{L}''$ with the edge set $\{(B(F),y) : F\in\tilde{\mathcal{F}} \text{ and } y\in D(F)\}$.
Denote by ${\mathcal{C}}=\{B_1, \ldots, B_p\}$ and $d_i=d_{\tilde{\mathcal{L}}}^+(B_i)$. Then $|\tilde{\mathcal{F}}|\le \sum_{i=1}^p{d_i\choose 	r-2}$. 
For $0\le d\le n$,
\[
{d\choose r-2}\le
\frac{d}{n}{n\choose r-2}.
\]
It follows that
\[
\sum_{i=1}^p d_i
\ge
\frac{n|\widetilde{\mathcal F}|}{{n\choose r-2}}.
\]
Therefore,
\[
\sum_{i=1}^p d_i
\ge \ell n-O(n^{7/9})
>(\ell-1)n
\]
in the $\mathcal L_k^r$-free case, for sufficiently large $n$.
Similarly,
\[
\sum_{i=1}^p d_i
\ge \ell'n-O(n^{7/9})
>(\ell'-1)n
\]
in the
$\{\mathcal L_k^r,2\mathcal L_{k-1}^r\}$-free case.

Assume $d_1\ge\ldots\ge d_p$.
\begin{clm}\label{CL: degree sum}
	$\sum_{i=1}^p d_i\le \ell n$  if $\mathcal{F}$ is $\mathcal{L}_k^r$-free and $\sum_{i=1}^p d_i\le \ell' n$  if $\mathcal{F}$ is $\{\mathcal{L}_k^r, 2\mathcal{L}_{k-1}^r\}$-free.
\end{clm}
    \begin{proof}
	Suppose $\mathcal{F}$ is $\mathcal{L}_k^r$-free and assume
	$\sum_{i=1}^p d_i>\ell n$. For every
	$(B(F),y)\in\tilde{\mathcal L}$ we have
	$y\notin V{\mathcal C}$. By the pigeonhole principle, some
	$v\in[n]\setminus V{\mathcal C}$ is the head of arcs from
	$\ell+1$ distinct pairs $B_{i_1},\ldots,B_{i_{\ell+1}}$.
	Note that $\ell=\phi(2,k)$. By Theorem~\ref{LEM: AHS}, the
	2-graph with edge set
	$\{B_{i_1},\ldots,B_{i_{\ell+1}}\}$ contains a $K_{1,k}$ or
	{a copy of $kK_2$}. For the former, $d^+(B_{i_j})>x$ permits a greedy
	extension to $\mathcal{L}_k^3$; for the latter,
	{if $M$ is the resulting matching of size $k$, then
	$\{e\cup\{v\}:e\in M\}$} is already a copy of $\mathcal{L}_k^3$. Both
	contradict Claim~\ref{CL: expansion}.
	
	Now suppose $\mathcal{F}$ is
	$\{\mathcal{L}_k^r,2\mathcal{L}_{k-1}^r\}$-free and assume
	$\sum_{i=1}^p d_i>\ell'n$. The same pigeonhole argument produces
	a vertex $v\in[n]\setminus V{\mathcal C}$ and $\ell'+1$ pairs.
	Theorem~\ref{LEM: 2K_{1,k-1}} produces one of
	$K_{1,k}$, {$kK_2$}, and $2K_{1,k-1}$. The first two give
	$\mathcal L_k^3$ as above, while the last greedily extends to
	$2\mathcal L_{k-1}^3$, again contradicting
	Claim~\ref{CL: expansion}.
   \end{proof}
\medskip

\begin{clm}\label{CL: degree bound}
	$d_1,\ldots,d_\ell\ge n-O(n^{\frac{7}{9}})$ if $\mathcal{F}$ is $\mathcal{L}_k^r$-free and $d_1,\ldots,d_{\ell'}\ge n-O(n^{\frac{7}{9}})$ if $\mathcal{F}$ is $\{\mathcal{L}_k^r, 2\mathcal{L}_{k-1}^r\}$-free.
\end{clm}
\begin{proof}
If $\mathcal{F}$ is $\mathcal{L}_k^r$-free, let $b=\sum_{i=1}^p d_i$ and $y=d_\ell$. Note that $n\ge d_1\ge \ldots\ge d_\ell=y\ge \ldots\ge d_p\ge 0$ and $(\ell-1)n<\sum_{i=1}^pd_i\le \ell n$. By the convexity of $f(x)={x\choose r-2}$, we have $$\sum_{i=1}^pf(d_i)\le (\ell-1)f(n)+qf(y)+f(y')$$ for some $0\le q\le p-\ell+1$ and $0\le y'<y$. So
$$\begin{aligned}
	\ell{ n\choose r-2}-O(n^{r-\frac{20}{9}})\le |\tilde{\mathcal{F}}|  &\le \sum_{i=1}^p{d_i\choose  r-2}\\
	&\le(\ell-1){n\choose r-2}+q{  y\choose r-2}+{  y'\choose r-2}\\
	&\le(\ell-1){n\choose r-2}+\frac{n}{y}{  y\choose r-2},
\end{aligned}$$
which implies that 
$y\ge n-O(n^{\frac{7}{9}})$.

With a similar argument, we can obtain that $d_1,\ldots,d_{\ell'}\ge n-O(n^{\frac{7}{9}})$ for the case that $\mathcal{F}$ is $\{\mathcal{L}_k^r, 2\mathcal{L}_{k-1}^r\}$-free.
 \end{proof}

\medskip

In case (I), let $G$ be the graph with edge set
$\{B_1,\ldots,B_\ell\}$; in case (II), let $G'$ be the graph with
edge set $\{B_1,\ldots,B_{\ell'}\}$.  
By Claim~\ref{CL: degree bound}, every selected pair $B_i$ has
\[
 |N^+(B_i)|=n-O(n^{7/9}).
\]
If the selected pairs contain a $K_{1,k}$, choose for its $k$ edges
distinct heads from their respective out-neighborhoods, avoiding all
vertices already used.  The resulting triples form an
$\mathcal L_k^3$.  If they contain a matching
$B_{i_1},\ldots,B_{i_k}$, then
\[
 \left|\bigcap_{j=1}^kN^+(B_{i_j})\right|
 \ge n-\sum_{j=1}^k\bigl(n-|N^+(B_{i_j})|\bigr)
 =n-O(n^{7/9})>2k
\]
for sufficiently large $n$.  Choosing a common head outside the
matching gives another $\mathcal L_k^3$.  Both conclusions contradict
Claim~\ref{CL: expansion}.  Hence $G$ is
$\{K_{1,k},kK_2\}$-free.

The same two arguments apply to $G'$.  If, in addition, $G'$ contains
two vertex-disjoint copies of $K_{1,k-1}$, choose distinct heads for
their $2(k-1)$ edges, each from the corresponding out-neighborhood
and outside all previously used vertices.  This gives two
vertex-disjoint copies of $\mathcal L_{k-1}^3$, again contradicting
Claim~\ref{CL: expansion}.  Therefore $G'$ is
$\{K_{1,k},kK_2,2K_{1,k-1}\}$-free.

Since $e(G)=\ell$ and $e(G')=\ell'$, the odd equality cases of
Theorems~\ref{LEM: AHS} and \ref{LEM: 2K_{1,k-1}} now give
\[
 G\cong2K_k
 \qquad\text{and}\qquad
 G'\in R(2k-1,k-1),
\]
respectively, after isolated vertices are discarded.

We next record the lower-order comparison that distinguishes the
members of $R(2k-1,k-1)$.  Given a graph $H$, put
\[
 \tau_i(H)=\bigl|\{X\in{V(H)\choose3}:e(H[X])=i\}\bigr|,
 \qquad i=1,2,3,
\]
and, for $V(H)\subseteq[n]$, define
\[
 \mathcal F_1(H;n,r)=
 \left\{F\in{[n]\choose r}:
 F\cap V(H)\ne\emptyset\ \text{and}\
 \delta\bigl(H[F\cap V(H)]\bigr)\ge1\right\}.
\]

\begin{thm}[Theorems 5.1 and 5.2 in~\cite{CF87}]
\label{THM: number of triangles}
Let $H\in R(2k-1,k-1)$.
If $k$ is odd, then
\[
 \tau_3(H)\ge
 \frac{k-1}{2}\left(\frac{k-1}{2}-1\right),
\]
with equality if and only if $H\cong G_1$.
\end{thm}

For a fixed graph $H$ on $h$ vertices, separating the members of
$\mathcal F_1(H;n,r)$ according to their intersection with $V(H)$
gives
\begin{equation}\label{EQ: graph-model-expansion}
\begin{aligned}
 |\mathcal F_1(H;n,r)|
  ={}& e(H){n-h\choose r-2}
   +\bigl(\tau_2(H)+\tau_3(H)\bigr)
      {n-h\choose r-3}
   +O(n^{r-4}).
\end{aligned}
\end{equation}
Indeed, a two-vertex trace contributes precisely when it is an edge
of $H$, a three-vertex trace contributes precisely when it spans two
or three edges, and all traces of size at least four contribute only
$O(n^{r-4})$ members.  Moreover, double-counting adjacent pairs of
edges gives
\[
 \tau_2(H)+3\tau_3(H)
   =\sum_{v\in V(H)}{d_H(v)\choose2},
\]
and hence
\begin{equation}\label{EQ: three-trace-count}
 \tau_2(H)+\tau_3(H)
 =\sum_{v\in V(H)}{d_H(v)\choose2}-2\tau_3(H).
\end{equation}
Consequently, if $H,H_0\in R(2k-1,k-1)$ and
$\tau_3(H_0)$ is minimum, then
\[
 |\mathcal F_1(H_0;n,r)|-|\mathcal F_1(H;n,r)|
 =2\bigl(\tau_3(H)-\tau_3(H_0)\bigr)
   {n-2k+1\choose r-3}+O(n^{r-4}).
\]
Thus, for all sufficiently large $n$, the size of 
$\mathcal F_1(H;n,r)$ is maximized over
$H\in R(2k-1,k-1)$ precisely by $H\cong G_1$.

{
We now complete cases (I) and (II) simultaneously.  Put
\[
 (t,H)=
 \begin{cases}
  (\ell,G),&\text{in case (I)},\\
  (\ell',G'),&\text{in case (II)},
 \end{cases}
 \qquad S=V(H),
\]
where isolated vertices are omitted, and write
\[
 \mathcal M_H=
 \left\{F\in{[n]\choose r}:F\cap S\ne\emptyset\ \text{and}\
 \delta\bigl(H[F\cap S]\bigr)\ge1\right\}.
\]
Thus $H\cong2K_k$ and $\mathcal M_H=\mathcal L_2(n,r,k,1)$ in
case (I), whereas $H\in R(2k-1,k-1)$ and
$\mathcal M_H=\mathcal F_1(H;n,r)$ in case (II).  In either case,
\begin{equation}\label{EQ: unified core properties}
 |S|=2k-1\ \text{or}\ 2k,\qquad
 d_H(x)=k-1\quad(x\in S),\qquad \nu(H)=k-1.
\end{equation}
For $H\cong2K_k$ these assertions are immediate.  In case (II), 
Vizing's theorem
partitions $E(H)$ into at most $\Delta(H)+1=k$ matchings, one of which
has size at least
\[
 \left\lceil\frac{e(H)}k\right\rceil
 =\left\lceil\frac{(2k-1)(k-1)}{2k}\right\rceil=k-1.
\]

For $i\in[t]$, define
\[
 P_i=\left\{y\in[n]\setminus S:
 d_{\mathcal F}^*(B_i\cup\{y\})\ge s\right\},
 \qquad
 M=\bigcap_{i=1}^{t}P_i,
 \qquad
 Z=[n]\setminus(S\cup M).
\]
Since $N^+(B_i)\setminus S\subseteq P_i$, Claim~\ref{CL: degree bound}
gives $|P_i|=n-O(n^{7/9})$ for every $i\in[t]$.  More precisely, the
definition of $Z$ yields
\begin{equation}\label{EQ: unified Z bound}
 |Z|=\left|([n]\setminus S)\setminus\bigcap_{i=1}^{t}P_i\right|
 \le\sum_{i=1}^{t}|([n]\setminus S)\setminus P_i|
 =O(n^{7/9}).
\end{equation}
In particular, $|M|=n-O(n^{7/9})$.  Since $Z\cap S=\emptyset$, for
every $z\in Z$ the fact that $z\notin M$ supplies an index
$i\in[t]$ such that
\[
 d_{\mathcal F}^*(B_i\cup\{z\})<s.
\]
Choose one such edge and denote it by $B_z$.  Define the missing-edge
family
\[
 \mathcal D'_H=
 \left\{B_z\cup\{z\}\cup X:
 z\in Z,\ X\in{M\choose r-3},\
 B_z\cup\{z\}\cup X\notin\mathcal F\right\}.
\]
For a fixed $z\in Z$, the family
\[
 \mathcal Q_z=\left\{X\in{M\choose r-3}:
 B_z\cup\{z\}\cup X\in\mathcal F\right\}
\]
has matching number less than $s$; otherwise its $s$ disjoint
members, after adjoining the common kernel $B_z\cup\{z\}$, contradict
the choice of $B_z$.  A maximal matching in $\mathcal Q_z$ has union
of size at most $(s-1)(r-3)$ and this union is a vertex cover of
$\mathcal Q_z$.  Hence
\[
 |\mathcal Q_z|\le
 (s-1)(r-3){|M|-1\choose r-4}.
\]
The members counted for distinct $z$ are distinct, because each of
them meets $Z$ exactly in $z$.  Consequently,
\begin{equation}\label{EQ: unified missing family}
 |\mathcal D'_H|\ge |Z|\left({|M|\choose r-3}
 -(s-1)(r-3){|M|-1\choose r-4}\right)
 =\Omega(|Z|n^{r-3}).
\end{equation}
On the other hand, $\mathcal F[Z]$ is $\mathcal L_k^r$-free in both
cases, so Theorem~\ref{THM: asymptotic} (with the bounded cases
absorbed into the constant) gives
\begin{equation}\label{EQ: unified exceptional family}
 |\mathcal F[Z]|=O(|Z|^{r-2}).
\end{equation}

We claim the following single containment for both cases:
\begin{equation}\label{EQ: unified containment}
 \mathcal F\subseteq
 (\mathcal M_H\setminus\mathcal D'_H)\cup\mathcal F[Z].
\end{equation}
Fix $F\in\mathcal F\setminus\mathcal F[Z]$ and put $T=F\cap S$.
Suppose first that $T=\emptyset$.  Since
$[n]=S\sqcup M\sqcup Z$ and $F\nsubseteq Z$,
there is a vertex $y\in F\cap M$.  By
\eqref{EQ: unified core properties}, choose a matching
$B_{i_1},\ldots,B_{i_{k-1}}$ in $H$.  For every $j\in[k-1]$,
\[
 d_{\mathcal F}^*(B_{i_j}\cup\{y\})\ge s.
\]
These triples form a linear $(k-1)$-star centered at $y$ and are disjoint from {$F\setminus\{y\}$}.   The greedy extension used in
Claim~\ref{CL: expansion}, with $F$ prescribed as the first petal,
therefore produces the remaining $k-1$ petals of an
$\mathcal L_k^r$, a contradiction.
Thus $T\ne\emptyset$.  If some $x\in T$ is isolated in $H[T]$, then
all its $k-1$ neighbors $u_1,\ldots,u_{k-1}$ lie outside $F$.  Choose
distinct $y_1,\ldots,y_{k-1}\in M\setminus F$, which is possible for
large $n$.  For each $j$,
\[
 d_{\mathcal F}^*(\{x,u_j,y_j\})\ge s.
\]
The resulting triples form a linear $(k-1)$-star centered at $x$,
and  avoid {$F\setminus\{x\}$}.  The same greedy extension together with
$F$ again gives an $\mathcal L_k^r$, a contradiction.  Therefore
$H[T]$ has no isolated vertex, and hence $F\in\mathcal M_H$.  Finally,
$\mathcal D'_H\cap\mathcal F=\emptyset$ by its definition, proving
\eqref{EQ: unified containment}.

Every member of $\mathcal D'_H$ lies in $\mathcal M_H$, because its
trace on $S$ is the edge $B_z$, whereas $\mathcal F[Z]$ is disjoint
from $\mathcal M_H$.  It follows that
\begin{equation}\label{EQ: unified comparison}
 |\mathcal F|\le
 |\mathcal M_H|-|\mathcal D'_H|+|\mathcal F[Z]|.
\end{equation}
In case (I), $|\mathcal M_H|=|\mathcal L_2(n,r,k,1)|=|\mathcal L|$.
In case (II), equations~\eqref{EQ: graph-model-expansion} and
\eqref{EQ: three-trace-count}, together with
Theorem~\ref{THM: number of triangles}, give
\[
 |\mathcal M_H|\le|\mathcal F_1(G_1;n,r)|
 =|\mathcal L_1(n,r,k,1)|=|\mathcal L|,
\]
with equality if and only if $H\cong G_1$.  If $Z\ne\emptyset$, then
\eqref{EQ: unified Z bound}--\eqref{EQ: unified exceptional family}
and $r\ge5$ imply
\[
 |\mathcal D'_H|-|\mathcal F[Z]|
 \ge c|Z|n^{r-3}-C|Z|^{r-2}>0
\]
for all sufficiently large $n$.  Equation~\eqref{EQ: unified comparison}
would then give $|\mathcal F|<|\mathcal L|$, contrary to the initial
assumption.  Hence $Z=\emptyset$.  The containment and the preceding
graph-model comparison now give
\[
 |\mathcal L|\le|\mathcal F|\le|\mathcal M_H|\le|\mathcal L|.
\]
Thus all inequalities are equalities.  In case (I), this yields
$\mathcal F=\mathcal M_H\cong\mathcal L_2(n,r,k,1)$.  In case (II),
it first yields $H\cong G_1$ and then
$\mathcal F=\mathcal M_H\cong\mathcal L_1(n,r,k,1)$.

For completeness, the claimed families do satisfy the forbidden
conditions.  In either core graph $H$ one has
$\Delta(H)=\nu(H)=k-1$.  A linear $k$-star in $\mathcal M_H$ centered
in $S$ would require $k$ distinct neighbors of its center in $H$,
while one centered outside $S$ would require a matching of size $k$
in $H$; both are impossible.  In case (II), a linear $(k-1)$-star
centered in $S$ uses at least $k$ vertices of $S$, and one centered
outside $S$ uses at least $2(k-1)>k$ vertices of $S$.  Hence two
vertex-disjoint such stars would use at least $2k>|S|=2k-1$ core
vertices.  Therefore the case-(II) construction is also
$2\mathcal L_{k-1}^r$-free.  This proves both theorems.
}
\end{proof}

\section{{Proof of Theorem~\ref{THM: Main}}}\label{SEC: proof main}

{We first verify the lower-bound constructions and then
prove the upper bound and its equality statement.}

\subsection{{The lower-bound constructions}}

\begingroup

\begin{lem}\label{LEM: tree lower bounds}
Let $T,A,B,a,k,A_0,B_0$ be as in Theorem~\ref{THM: Main}; in
particular, $r\ge5$ and $k\ge3$ is odd. Then:
\begin{enumerate}
 \item $\mathcal L_1(n,r,k,a)$ is $T^{(r)}$-free.
 \item If $B_0=\emptyset$, then $\mathcal L_2(n,r,k,a)$ is
       $T^{(r)}$-free.
 \item If $B_0\ne\emptyset$ and
       $b=\min_{y\in B_0}d_T(y)-2\ge1$, then
       $\mathcal L_2(n,r,k,a;R_{b-1})$ is $T^{(r)}$-free for every
       choice of $R_{b-1}\in R(a-1,b-1)$ on $W$.
\end{enumerate}
\end{lem}

\begin{proof}

Suppose one of these constructions contains a copy $\mathcal H$ of $T^{(r)}$. Identify the original vertices of $T$ with their images in $\mathcal H$. Let $G_i$ be the core graph of the corresponding construction for $i\in\{1,2\}$, and set $S=V(G_i)$.

Since $k$ is odd, $\mathcal F_3=\emptyset$. Moreover,
\[
 |S|\le 2k,\qquad \Delta(G_i)=k-1,\qquad \nu(G_i)=k-1.
\]
Every edge of the construction that avoids $W$ has a nonempty trace
on $S$, and the induced graph on this trace has no isolated vertices.
In particular, such a trace contains an edge of $G_i$.

We first define the auxiliary graph that will be used throughout the
proof. The vertices of $T$ are called the original vertices of $T^{(r)}$ (also of $\mathcal{H}$), and the vertices added during the expansion are called the private vertices. 

For every expanded edge $F$ of $\mathcal H$, we choose a
two-element set $f_F\subseteq F$ as follows.

Retain every original vertex of $F$ that belongs to $W\cup S$.

\begin{itemize}
\item If two original vertices of $F$ belong to $W\cup S$, let
$f_F$ be these two vertices.

\item Suppose that exactly one original vertex $v$ of $F$ belongs to
$W\cup S$. If $v\in W$, choose any private vertex of $F$ and put it
together with $v$ in $f_F$.

If $v\in S$ and $F\cap W\neq\emptyset$, choose a vertex of
$F\cap W$ as the second vertex of $f_F$. This vertex is private with
respect to the copy of $T$ because $W\cap S=\emptyset$.

If $v\in S$ and $F\cap W=\emptyset$, then
$G_i[F\cap S]$ contains an edge. Choose a neighbor of $v$ in this
graph and let $f_F$ consist of $v$ and this neighbor.

\item Suppose that no original vertex of $F$ belongs to $W\cup S$.
If $F\cap W\neq\emptyset$, choose one vertex of $F\cap W$ and one
additional private vertex of $F$.

If $F\cap W=\emptyset$, then $F\cap S\neq\emptyset$ and
$G_i[F\cap S]$ has no isolated vertices. Choose an edge of
$G_i[F\cap S]$ and use it as $f_F$.
\end{itemize}

There are always enough private vertices since $r-2\ge2$.

For a graph $G$ and a vertex $v$, we say that splitting $v$ means
replacing $v$ by distinct vertices $v_e$, one for each edge $e$
incident with $v$, and replacing the edge $vu$ by $v_eu$. Thus every
new vertex $v_e$ is a leaf. Splitting several vertices is performed
simultaneously.

Let $\widetilde T$ be the graph with vertex set
\[
Y=\bigcup_{F\in E(\mathcal H)}V(f_F)
\]
and edge set
\[
E(\widetilde T)=\{f_F:F\in E(\mathcal H)\}.
\]

Every original vertex belonging to $W\cup S$ is retained in $f_F$
for its incident expanded edge. Every original vertex outside
$W\cup S$ is represented by private vertices in its incident expanded edges. Consequently, $\widetilde T$ is
obtained from $T$ by splitting every original vertex outside
$W\cup S$, possibly followed by deleting isolated vertices or leaves.
In particular, $\widetilde T$ is a forest.

Also, by the choice of the pairs $f_F$, every vertex of $Y$ that
belongs to $F$ is one of the two vertices selected for $f_F$. Hence $Y$ is a $2$-crosscut of the copy $\mathcal H$.

Let $T'$ be the graph consisting of the pairs $f_F$ that avoid $W$.
By construction, every edge of $T'$ lies in $G_i$, and therefore
\[
\Delta(T')\le k-1,\qquad \nu(T')\le k-1.
\]
Indeed, the edges of $T'$ incident with a fixed vertex $v$ use
distinct neighbors of $v$ in $G_i$, while a matching in $T'$ gives
a matching of the same size in $G_i$.

Let
\[
\mathcal P=\{f_F\in E(T'):F\cap W\neq\emptyset\},
\qquad p=|\mathcal P|,
\]
and let
\[
T_0=T'-\mathcal P.
\]
For every pair in $\mathcal P$, both endpoints are original vertices
in $S$, while $F\cap W$ contains a private vertex not belonging to
$Y$. These private vertices are distinct for different members of
$\mathcal P$. Consequently,
\[
|W\cap Y|\le a-1-p. \tag{1}
\]

We next show that no vertex of $A$ is split. Suppose that some
$x\in A$ is split. The split copies of $x$ can be
matched to all the distinct neighbors of $x$, giving \(d_T(x)\)
matching edges. Every other vertex of \(A\) has at least one child in
\(B\), since its degree is at least \(k\ge3\). Choose one such child
for each vertex of \(A\setminus\{x\}\). These children are distinct
and none is a neighbor of the root \(x\). Hence the resulting graph
contains a matching of size
\[
d_T(x)+a-1\ge k+a-1. \tag{2}
\]

On the other hand, every matching of $\widetilde T$ contains at most
\(k-1\) edges from \(T_0\), at most \(p\) edges from \(\mathcal P\),
and at most \(|W\cap Y|\) further edges. By (1),
\[
\nu(\widetilde T)
 \le (k-1)+p+|W\cap Y|
 \le a+k-2,
\]
contradicting (2). Therefore no vertex of \(A\) is split. Thus every
vertex of \(A\) belongs to \(Y\) and retains its original degree.

Put
\[
\widetilde B=Y\setminus A,\qquad
U=A\setminus W,\qquad u=|U|,\qquad
Q=\widetilde B\cap W,\qquad h=|Q|.
\]
Since \(T'\) consists of the pairs avoiding \(W\), all its edges join
\(U\) to \(\widetilde B\setminus W\). From (1), we obtain
\[
a-u+h=|W\cap Y|\le a-1-p,
\]
and hence
\[
u\ge h+p+1. \tag{3}
\]

If \(h=0\), every vertex of \(U\) has degree at least \(k\) in
\(T'\), so
\[
e(T')\ge ku.
\]
If \(h\ge1\), then the graph induced by \(U\cup Q\) is a forest and
therefore has at most \(u+h-1\) edges. Hence
\[
e(T')
 \ge ku-e_{\widetilde T}(U,Q)
 \ge ku-(u+h-1)
 =(k-1)u-h+1. \tag{4}
\]
Since \(T'\) is a forest, whenever the right-hand side of (4) is
positive,
\[
|V(T')|\ge e(T')+1.
\]

Suppose first that \(p\ge1\). If \(h=0\), then (3) gives \(u\ge2\),
and therefore
\[
|V(T')|\ge ku+1\ge2k+1>|S|.
\]
If \(h\ge1\), then \(u\ge h+2\), and (4) gives
\[
|V(T')|
 \ge (k-1)(h+2)-h+2
 =2k+(k-2)h
 >2k\ge |S|.
\]
Both cases are impossible. Thus
\[
p=0.
\]

Now \(T'=T_0\). The case \(h=0\) is impossible, since every vertex
of \(U\) would then have degree at least \(k\) in \(T'\), contradicting
\(\Delta(T')\le k-1\).

Suppose that \(h\ge2\). By (3), \(u\ge h+1\), and (4) gives
\[
|V(T')|
 \ge (k-2)h+k+1.
\]
If \(k\ge5\), this is at least \(3k-3>2k\), again impossible.

It remains to consider \(k=3\). The only possible case is
\(h=2\), \(u=3\), and
\[
|V(T')|=6,\qquad e(T')=5.
\]
Thus \(S=V(G_2)\) and \(G_2=2K_3\). Since all six core vertices
belong to \(Y\), the identity \(F\cap Y=f_F\) implies that every
edge of \(T'\) is an edge of \(2K_3\). But a forest contained in two
vertex-disjoint triangles has at most four edges, contradicting
\(e(T')=5\).

Therefore \(h=1\). Write \(Q=\{y\}\). By (3), \(u\ge2\). For every
\(v\in U\), all neighbors of \(v\) in \(\widetilde T\) lie either in
\(\widetilde B\setminus W\) or at \(y\). If \(v\) were not adjacent to
\(y\), then
\[
d_{T'}(v)=d_{\widetilde T}(v)=d_T(v)\ge k,
\]
contrary to \(\Delta(T')\le k-1\). Hence every vertex of \(U\) is
adjacent to \(y\), and
\[
d_T(v)=k,\qquad d_{T'}(v)=k-1.
\]

The remaining \(k-1\) neighbors of distinct vertices of \(U\) are
pairwise disjoint, since otherwise \(\widetilde T\) would contain a
cycle. Thus \(T'\) is the disjoint union of \(u\) copies of
\(K_{1,k-1}\), and consequently it uses \(ku\ge2k\) vertices.

For \(G_1\), we have \(|S|=2k-1\), which is impossible. This proves
assertion (1).

For \(G_2\), we must have \(u=2\). Let \(U=\{x_1,x_2\}\). The vertex
\(y\) has degree at least two in \(\widetilde T\), so \(y\) is an
unsplit original vertex of \(B\). Moreover, \(x_1,x_2\in A_0\), and
therefore \(y\in B_0\). This contradicts \(B_0=\emptyset\), proving
assertion (2).

Finally, suppose that \(B_0\neq\emptyset\), and put
\[
b=\min_{z\in B_0}d_T(z)-2\ge1.
\]
Since \(d_T(z)\le a\), we have \(b-1\le a-2\), so a graph
\(R_{b-1}\in R(a-1,b-1)\) exists.

In the case \(G_2\), we have \(u=2\) and \(h=1\). Thus all vertices
of \(W\) are occupied by
\[
\{y\}\cup(A\setminus\{x_1,x_2\}).
\]
For every neighbor \(z\in N_T(y)\setminus\{x_1,x_2\}\), the
corresponding expanded edge has trace on \(W\) exactly \(\{y,z\}\).
By the definition of \(S_W(R_{b-1})\), this pair must be an edge of
\(R_{b-1}\). Hence \(y\) has at least
\[
d_T(y)-2\ge b
\]
neighbors in \(R_{b-1}\), contradicting
\[
\Delta(R_{b-1})\le b-1.
\]
This proves assertion (3).
\end{proof}

\subsection{{The upper bound and equality case}}

Given a hypergraph $\mathcal{F}$ and a vertex $x\in V(\mathcal{F})$, let $\mathcal{L}_{\mathcal{F}}(x)=\{F\backslash\{x\} : x\in F\in \mathcal{F}\}$.
Given a set $A\subseteq V(\mathcal{F})$, let $\mathcal{L}_{\mathcal{F}}(A)=\cap_{x\in A}\mathcal{L}_{\mathcal{F}}(x)$. In other words,
$$\mathcal{L}_{\mathcal{F}}(A)=\{D\subseteq V(\mathcal{F})\backslash A : \, D\cup\{a\}\in \mathcal{F} \text{ for any } a\in A\}.$$
We call $\mathcal{L}_{\mathcal{F}}(A)$ the {\it common link hypergraph} of $A$ in $\mathcal{F}$.
%{This common link is an intersection of single-vertex links; it differs from $\operatorname{L}_{\mathcal F}(A)$, whose members complete the whole set $A$ to an edge.}
Let $\mathcal{K}_p^{(p)}(s)$ denote the complete $p$-partite $p$-graph with $s$ vertices in each part.

\begingroup
The following formulation combines the structural stability theorem
of F\"uredi and Jiang with their comparison lemma.  We include the
argument for the comparison and its equality case.

\begin{lem}[Theorem 4.2 and Lemma 9.3 in~\cite{FJ15}]
\label{THM: construction}
Let $r\ge4$, and let $\mathcal H$ be an $r$-graph on $s$ vertices
embeddable in a $2$-reducible $r$-tree.  Suppose that
$\sigma(\mathcal H)=\sigma\ge2$, and put
\[
 \beta=\frac{1}{(r-2)(\sigma+1)+1}.
\]
For every fixed $\gamma>0$, there exists $n_0=n_0(r,s,\gamma)$
such that the following holds whenever $n\ge n_0$.
Let $\mathcal F\subseteq{[n]\choose r}$ be $\mathcal H$-free with
\[
 |\mathcal F|\ge(\sigma-1){n\choose r-1}-n^{r-1-\beta}.
\]
(1) Then there exists $A\subseteq[n]$ with $|A|=\sigma-1$ such that
\begin{equation}\label{EQ: construction common link}
 |\mathcal L_{\mathcal F}(A)|
 \ge {n\choose r-1}-(r-1)(1+2s^2)n^{r-1-\beta}.
\end{equation}
(2) For any such $A$, let $\mathcal L^*\subseteq\mathcal L_{\mathcal F}(A)$
consist of all edges that lie in a copy of
$\mathcal K_{r-1}^{(r-1)}(s)$ contained in
$\mathcal L_{\mathcal F}(A)$.
{We take its vertex set to be its support,
$V(\mathcal L^*)=\bigcup_{D\in\mathcal L^*}D$, so no isolated
vertices are adjoined to $\mathcal L^*$.} Define
\[
 \mathcal F_A=\{F\in\mathcal F:F\cap A\ne\emptyset\},
 \qquad
 S_A=\left\{F\in{[n]\choose r}:F\cap A\ne\emptyset\right\},
\]
and let $\mathcal B\subseteq\mathcal F\setminus\mathcal F_A$.
Suppose either $\mathcal B=\emptyset$, or there exists an
$(r-1)$-graph $\mathcal D$ on $[n]\setminus A$ such that
\[
 \partial_{r-2}(\mathcal D)\cap\partial_{r-2}(\mathcal L^*)
 =\emptyset
 \qquad\text{and}\qquad
 |\mathcal D|\ge\gamma|\mathcal B|.
\]
Then
\[
 |\mathcal F_A\cup\mathcal B|
 \le |S_A|={n\choose r}-{n-\sigma+1\choose r}.
\]
Equality holds if and only if $\mathcal B=\emptyset$ and
$\mathcal F_A=S_A$.
\end{lem}
The existence of $A$ follows from Theorem~4.2 in~\cite{FJ15},
and the inequality follows from Lemma~9.3 therein.
Moreover, the proof of Lemma~9.3 gives a strict inequality
when $\mathcal B\ne\emptyset$ and $n$ is sufficiently large.
When $\mathcal B=\emptyset$, equality holds precisely when
$\mathcal F_A=S_A$.

\endgroup

{

To apply Lemma~\ref{THM: construction}, the residual family
$\mathcal B$ will be controlled as follows: we take
$\mathcal D=\partial_{r-1}(\mathcal B)$, so that
$\partial_{r-2}(\mathcal D)\cap\partial_{r-2}(\mathcal L^*)=\emptyset$, and $|\mathcal D|$ is bounded by the following lemma:

\begin{lem}[\cite{FJ15}]\label{LEM:tree-shadow bound}
Let $\mathcal H$ be an $r$-tree with $p$ vertices, and let
$\mathcal F$ be an $r$-graph on $[n]$ containing no copy of
$\mathcal H$. Then
\[
|\mathcal F|\le(p-r)\,|\partial_{r-1}(\mathcal F)|.
\]
In particular,
\[
\mathrm{ex}_r\big(n,\mathcal H\big)\le(p-r)\binom{n}{r-1}.
\]
\end{lem}

By lemma~\ref{LEM:tree-shadow bound}, we have that \[
|\partial_{r-1}(\mathcal{F}_{\le 1})| \ge \frac{|\mathcal{F}_{\le 1}|}{s},
\qquad s:=|V(T^{(r)})|-r.
\]

We apply Lemma~\ref{THM: construction} to expanded trees.  These are
$2$-reducible for $r\ge4$, as noted above.

\begin{lem}\label{LEM: constructions}
	
	Let $r\ge4$, let $T$ be a tree, and suppose $T^{(r)}$ has $s$
	vertices and $\sigma(T^{(r)})=\sigma\ge2$.  Put
	$\beta=((r-2)(\sigma+1)+1)^{-1}$, and let
	$n\ge n_0(r,s,1/s)$ be sufficiently large.
	Let $\mathcal F\subseteq{[n]\choose r}$ be $T^{(r)}$-free with
	$|\mathcal F|\ge(\sigma-1){n\choose r-1}-n^{r-1-\beta}$.
	Then the following statements hold.
	
	(a) There is a $(\sigma-1)$-set $W$ satisfying
	
	\[
	 |\mathcal L_{\mathcal F}(W)|\ge
	 {n\choose r-1}-(r-1)(1+2s^2)n^{r-1-\beta}.
	\]
	
	(b) Let $\mathcal{L}^*\subseteq \mathcal{L}_{\mathcal{F}}(W)$ consist of all edges of $\mathcal{L}_{\mathcal{F}}(W)$ that lie in copies of $\mathcal{K}_{r-1}^{(r-1)}(s)$ contained in $\mathcal L_{\mathcal F}(W)$, {with $V(\mathcal L^*)=\bigcup_{D\in\mathcal L^*}D$ as above}. Let $\mathcal{F}_W=\{F\in\mathcal{F} : F\cap W\neq\emptyset\}$. 
	Define
	$${\mathcal{F}_{\le 1}}=\{F\in\mathcal{F}\backslash\mathcal{F}_W:|F\cap V(\mathcal{L}^*)|\le 1\}$$
	$${\mathcal{F}_{\ge 2}}=\{F\in\mathcal{F}\backslash\mathcal{F}_W:|F\cap V(\mathcal{L}^*)|\ge 2\}.$$
	Then $$|\mathcal{F}_W\cup{\mathcal{F}_{\le 1}}|\le|S_W|= {n\choose r}-{n-\sigma+1\choose r}.$$
	Equality holds if and only if
	${\mathcal{F}_{\le 1}}=\emptyset$ and $\mathcal{F}_W=S_W$.

\end{lem}

\begin{proof}
Part~(a) is the existence assertion in
Lemma~\ref{THM: construction} (1).

For part~(b), let
$\mathcal{D}=\partial_{r-1}({\mathcal{F}_{\le 1}})$. Since
${\mathcal{F}_{\le 1}}$ is $T^{(r)}$-free, lemma~\ref{LEM:tree-shadow bound} gives
$|\mathcal{D}|\ge |{\mathcal{F}_{\le 1}}|/s$. Every member of
${\mathcal{F}_{\le 1}}$ contains at most one vertex of
$V(\mathcal{L}^*)$. Hence every member of
$\partial_{r-2}(\mathcal{D})$ contains at most one such vertex,
whereas every member of
$\partial_{r-2}(\mathcal{L}^*)$ is contained in
$V(\mathcal{L}^*)$. Since $r\ge4$, the two shadows are disjoint.
Lemma~\ref{THM: construction}, with $\gamma=1/s$, now gives the
claimed inequality.

Moreover, by Lemma~\ref{THM: construction} (2),  equality 
holds if and only if ${\mathcal{F}_{\le 1}}=\emptyset$ and $\mathcal{F}_W=S_W$.
\end{proof}

\begingroup

\begin{proof}[Proof of Theorem~\ref{THM: Main}]
We now prove the upper bound and its equality statement; the lower
bounds have already been established in Lemma~\ref{LEM: tree lower bounds}.
Put $\mathcal H=T^{(r)}$, $q=e(T)$, and
$s=|V(\mathcal H)|=a+|B|+(r-2)q$.

We first claim that $ \sigma(\mathcal H)=a$.  For every nonempty $U\subseteq A$, the graph
$T[U\cup N_T(U)]$ is a forest, so
\[
 k|U|\le e\bigl(T[U\cup N_T(U)]\bigr)
 \le |U|+|N_T(U)|-1.
\]
Hall's theorem therefore gives a matching saturating $A$.
Its $a$ expanded edges are pairwise disjoint, whereas $A$ itself
is a crosscut of $\mathcal H$. Hence
\[
 a\le\tau(\mathcal H)\le\sigma(\mathcal H)\le a,
 \qquad\text{so }\sigma(\mathcal H)=a.
\]
If $a=1$, then $T=K_{1,k}$, $B_0=\emptyset$, and the exact
conclusion and uniqueness follow from Theorem~\ref{LEM: linear star}.
Thus assume $a\ge2$.

Let $\mathcal F\subseteq{[n]\choose r}$ be $\mathcal H$-free.
For the upper bound we may assume
$|\mathcal F|\ge|\mathcal L_2(n,r,k,a)|$; otherwise there is
nothing to prove. For a set $W$ of size $a-1$,
\[
 |S_W|={n\choose r}-{n-a+1\choose r}
       =(a-1){n\choose r-1}+O(n^{r-2}).
\]
Since $\beta=((r-2)(a+1)+1)^{-1}<1$, this implies
$|\mathcal F|\ge(a-1){n\choose r-1}-n^{r-1-\beta}$ for
sufficiently large $n$. Also $\mathcal H$ is $2$-reducible.
Lemma~\ref{LEM: constructions} now provides a set $W$ with
$|W|=a-1$ and the families
$\mathcal L^*,\mathcal F_W,{\mathcal F_{\le 1}},{\mathcal F_{\ge 2}}$ defined there.
In particular, $\mathcal F_W,{\mathcal F_{\le 1}},{\mathcal F_{\ge 2}}$ partition
$\mathcal F$, and the last two families avoid $W$.

\begin{clm}\label{LEM: linear star free}
The family ${\mathcal F_{\ge 2}}$ is $\mathcal L_k^r$-free.
\end{clm}
\begin{proof}
Choose $x\in A_0$ and write $N_T(x)=\{b_1,\ldots,b_k\}$.
The expanded edges through $x$ are
$E_i=\{x,b_i\}\cup Q_i$, where $|Q_i|=r-2$.
Suppose that $F_1,\ldots,F_k\in{\mathcal F_{\ge 2}}$ form a linear
$k$-star with center $v$.
For each $i$, the definition of ${\mathcal F_{\ge 2}}$ gives
\[
 |F_i\cap V(\mathcal L^*)|\ge2.
\]
Choose $y_i\in(F_i\setminus\{v\})\cap V(\mathcal L^*)$.
The $y_i$ are distinct because different petals meet only at $v$.
Define an injection $\phi$ on this expanded star by
$\phi(x)=v$, $\phi(b_i)=y_i$, and any bijection
from $Q_i$ to $F_i\setminus\{v,y_i\}$.
Thus even the leaf neighbors of $x$ have already been assigned images.

Map $A\setminus\{x\}$ bijectively onto $W$.
Next map $B\setminus N_T(x)$ to distinct unused vertices of
$V(\mathcal L^*)$.
These choices are possible: the existence of a vertex $y_i$ implies
that $\mathcal L^*$ contains a copy of
$\mathcal K_{r-1}^{(r-1)}(s)$, so
$|V(\mathcal L^*)|\ge(r-1)s$.
At this point the number of occupied vertices is
\[
 1+k(r-1)+(a-1)+(|B|-k)
 =a+|B|+k(r-2)\le s.
\]
For each remaining edge $ub\in E(T)$, where $u\in A\setminus\{x\}$
and $b\in B$, choose a complete multipartite witness
$\mathcal K_b\subseteq\mathcal L^*$ containing $\phi(b)$.
Choose one unused vertex from each of the $r-2$ parts of
$\mathcal K_b$ other than the part containing $\phi(b)$.
Together with $\phi(b)$ they form an edge $D_{ub}\in\mathcal L^*$.
Since $\mathcal L^*\subseteq\mathcal L_{\mathcal F}(W)$,
\[
 \{\phi(u)\}\cup D_{ub}\in\mathcal F.
\]
Every chosen vertex must avoid all previously assigned vertices,
including the images of all of $A\cup B$ and all earlier private
vertices. Before each such choice fewer than $s$ vertices have
been occupied, while every part has $s$ vertices. Thus the choices
can always be made, even when different witnesses overlap.
This extends $\phi$ to an embedding of  $T^{(r)}$ into
$\mathcal F$, a contradiction.
\end{proof}

Put $N=n-a+1$, the number of vertices outside $W$.
By the claim and the odd-$k$ statement of
Theorem~\ref{LEM: linear star},
\[
 |{\mathcal F_{\ge 2}}|\le|\mathcal L_2(N,r,k,1)|,
\]
with equality if and only if
${\mathcal F_{\ge 2}}\cong\mathcal L_2(N,r,k,1)$, on the $N$ vertices
outside $W$. Lemma~\ref{LEM: constructions}(b) yields
\begin{align*}
 |\mathcal F|
 &=|\mathcal F_W\cup{\mathcal F_{\le 1}}|+|{\mathcal F_{\ge 2}}|\\
 &\le|S_W|+|\mathcal L_2(N,r,k,1)|
 =|\mathcal L_2(n,r,k,a)|.
\end{align*}
Equality forces equality in both inequalities, namely
\[
 {\mathcal F_{\le 1}}=\emptyset,\qquad \mathcal F_W=S_W,\qquad
 {\mathcal F_{\ge 2}}\cong\mathcal L_2(N,r,k,1).
\]
Consequently every $\mathcal H$-free family attaining this upper
bound is isomorphic to $\mathcal L_2(n,r,k,a)$.
This is a conditional equality statement: it does not assert
that this construction is $\mathcal H$-free for every $T$.

For $B_0=\emptyset$, Lemma~\ref{LEM: tree lower bounds} supplies
a matching lower bound. Thus the upper bound is attained, and the
preceding equality characterization gives uniqueness. Together
with the lower bounds proved in that lemma, this completes the
proof of Theorem~\ref{THM: Main}.
\end{proof}
\endgroup

\section{Remarks and discussions}\label{SEC: remarks}
\begingroup

We give a counterexample showing that the odd-$k$ conclusions cannot
be extended to even $k$ simply by using
$\mathcal L_1(n,r,k,1)$ as the extremal construction.

Let $r\ge5$,  $k=4$, and let $n$ be sufficiently large. In Construction~A,
take $k=4$ and $a=1$, so that $W=\emptyset$, and put
\[
 S=V(G_1)=\{u_1,u_2,u_3,v_1,v_2,v_3,v_4\}.
\]
Thus $G_1$ is obtained from
$K[\{u_1,u_2,u_3\},\{v_1,v_2,v_3\}]$ by deleting $u_1v_1$
and adding $u_1v_4$ and $v_1v_4$. Every vertex of $G_1$ has
degree three except $v_4$, which has degree two. Define
\[
 D=\{v_4,u_2,u_3,v_2\},\qquad
 {\mathcal B}=\left\{D\cup X:
 X\in{V\setminus S\choose r-4}\right\},
\]
and put
\[
 \mathcal L^+=\mathcal L_1(n,r,4,1)\cup{\mathcal B}.
\]
The vertex $v_4$ is isolated in $G_1[D]$, and $D$ contains
neither of the special triples $\{v_4,u_j,v_3\}$, $j=2,3$.
Hence ${\mathcal B}$ is disjoint from $\mathcal L_1(n,r,4,1)$, and
\[
 |\mathcal L^+|=|\mathcal L_1(n,r,4,1)|+{n-7\choose r-4}.
\]

\begin{clm}
$\mathcal L^+$ is
$\{\mathcal L_4^r,2\mathcal L_3^r\}$-free.
\end{clm}
\begin{proof}First consider $\mathcal L_1(n,r,4,1)=\mathcal F_1\cup\mathcal F_3$.
Every edge meets $S$ in at least two vertices, so a linear
$4$-star centered outside $S$ would require at least eight
vertices of $S$, which is impossible. For a center $x\in S$,
each $\mathcal F_1$-petal contains a neighbor of $x$ in $G_1$;
these neighbors must be distinct. All $\mathcal F_3$-edges contain
$v_4$ and $v_3$, so at most one can occur in a linear star.
If $x=v_4$, there are at most two $\mathcal F_1$-petals and
one $\mathcal F_3$-petal. Suppose that $x\ne v_4$ and that a petal $F$ belongs to
$\mathcal F_3$. By definition, $F$ contains
$\{v_4,u_j,v_3\}$ for some $j\in\{2,3\}$ but contains neither
$u_1$ nor $v_1$. Since $x\in F\cap S$ and $x\ne v_4$, we have
$x\in\{u_2,u_3,v_2,v_3\}$. If $x\in\{u_2,u_3\}$, then
$v_3\in F$ is a neighbor of $x$; if $x\in\{v_2,v_3\}$, then
$u_j\in F$ is a neighbor of $x$. Hence
$|N_{G_1}(x)\setminus F|\le2$, as $d_{G_1}(x)=3$.
Every other petal belongs to $\mathcal F_1$, meets $F$ only in
$x$, and contains a neighbor of $x$ outside $F$. These neighbors
must be distinct, since any two petals intersect only in $x$.
Therefore, at most two other petals can occur.
Without an $\mathcal F_3$-petal, there are at most
$d_{G_1}(x)\le3$ petals. Thus $\mathcal L_1(n,r,4,1)$
is $\mathcal L_4^r$-free.

Suppose next that a linear $4$-star in $\mathcal L^+$ contains
$F=D\cup X\in{\mathcal B}$. It contains only one member of
${\mathcal B}$, since all such members contain $D$.
Every $\mathcal F_3$-edge meets $D$ in at least two vertices,
so the other three petals must belong to $\mathcal F_1$.
If the center is $v_4$, only its two core neighbors are available.
If the center is $u_2$ or $u_3$, the neighbor $v_2\in D$
is blocked, again leaving only two available core neighbors.
If the center is $v_2$, both $u_2$ and $u_3$ are blocked,
leaving only one. Finally, if the center lies in $X$, the
other three petals require pairwise disjoint core traces of
size at least two inside $S\setminus D$, which has only three
vertices. All cases are impossible, proving that $\mathcal L^+$
is $\mathcal L_4^r$-free.

To exclude two vertex-disjoint linear $3$-stars, recall that
every edge of $\mathcal L^+$ meets $S$ in at least two vertices.
A linear $3$-star centered in $S$ therefore uses at least four
vertices of $S$, while one centered outside $S$ uses at least six.
Two vertex-disjoint copies would use at least eight vertices of
$S$, contradicting $|S|=7$.
\end{proof}

Consequently,
\[
 \begin{aligned}
 \ex_r(n,\mathcal L_4^r)
 &\ge \ex_r\bigl(n,\{\mathcal L_4^r,2\mathcal L_3^r\}\bigr)\\
 &\ge |\mathcal L_1(n,r,4,1)|+{n-7\choose r-4}
 >|\mathcal L_1(n,r,4,1)|.
 \end{aligned}
\]
%In particular, when $r=5$, the additional edges are $\{v_4,u_2,u_3,v_2,x\}$ with $x\in V\setminus S$, giving
%$n-7$ extra edges. 
Taking $T=K_{1,4}$ also gives a counterexample
to the corresponding even-$k$ extension of
Theorem~\ref{THM: Main}. Thus additional core traces are necessary
in the even-$k$ case; this example does not assert that
$\mathcal L^+$ is itself extremal.

%\medskip
The example above motivates the following problems concerning the
remaining cases of Theorems~\ref{LEM: linear star},
\ref{LEM: two linear star}, and~\ref{THM: Main}.

%\medskip
\begin{problem}
%\noindent\textbf{Problem 1.}
(a) For  even $k\ge4$, determine  the exact Tur\'an number  $\ex_r(n,\mathcal L_k^r)$,  $ \ex_r\bigl(n,\{\mathcal L_k^r,2\mathcal L_{k-1}^r\}\bigr)$,   $\ex_r(n,T^{(r)})$ and all extremal families for sufficiently large $n$. 

(b) For odd $k\ge3$ and $B_0\ne\emptyset$, determine the exact
Tur\'an number  $\ex_r(n,T^{(r)})$ and all extremal families. 

\end{problem}
    
\medskip
%\noindent{\bf Acknowledgment:}
\section*{Acknowledgments} This work was supported by the National Key R\&D Program of China (2023YFA1010203), the
National Natural Science Foundation of China (12471336, 12601648), 
the Quantum Science and Technology-National Science and Technology Major Project (2021ZD0302902).
\section*{Declaration of generative AI}
%and AI-assisted technologies in the manuscript preparation process}

An initial draft of this manuscript was completed in late 2024.
During the preparation of the present version, the authors used
GPT-6 Astra to assist in constructing the counterexample
for $k=4$, preparing the TikZ code for Figure~1, and to improve the language and readability of the manuscript.
The authors verified the counterexample and its mathematical
justification, reviewed and edited the AI-assisted material, and
take full responsibility for the accuracy and integrity of the
entire manuscript.

\end{document}